\documentclass[10pt,a4paper]{amsart}

\usepackage{amsmath}

\usepackage{amssymb}
\usepackage{indentfirst}
\usepackage[latin1]{inputenc}
\usepackage{geometry}
\usepackage{amsfonts}
\usepackage{enumitem}
\usepackage{a4wide}
\usepackage[english]{babel}
\usepackage{longtable}
\usepackage{url}
\usepackage[all]{xy}
\usepackage{nicefrac}
\usepackage{amsfonts}
\usepackage{verbatim}
\usepackage{subfig}
\usepackage{cancel}

\usepackage{tikz}
\usetikzlibrary{matrix}
\usetikzlibrary{arrows}
\usetikzlibrary{fit}
\usetikzlibrary{shapes}

\usepackage{booktabs}

\usepackage{amsthm}

\DeclareSymbolFont{largesymbols}{OMX}{yhex}{m}{n}
\DeclareMathAccent{\widehat}{\mathord}{largesymbols}{"62}

\newcommand{\Q}{{\mathbb Q}}
\newcommand{\Z}{{\mathbb Z}}
\newcommand{\C}{{\mathbb C}}
\newcommand{\N}{{\mathbb N}}
\newcommand{\R}{{\mathbb R}}

\newcommand{\U}{\mathcal{U}}
\newcommand{\OO}{\mathcal{O}}
\newcommand{\Gal}{\textnormal{Gal}}
\newcommand{\Cen}{\textnormal{Cen}}

\newcommand{\GL}{\textnormal{GL}}
\newcommand{\SL}{\textnormal{SL}}

\newcommand{\ind}{\textnormal{ind}}
\newcommand{\tr}{\textnormal{tr}}

\newcommand{\suma}[1]{\widehat{#1}}

\newcommand{\inv}{{^{-1}}}
\newcommand{\GEN}[1]{\langle #1 \rangle}

\newcommand{\iso}{\cong}

\newcommand{\Aut}{\operatorname{Aut}}

\newcommand{\Irr}{{\operatorname{Irr}}}

\newcommand{\quat}[2]{\left(\frac{#1}{#2}\right)}

\newcommand{\lcm}{\textnormal{lcm}}
\newcommand{\Fix}{\textnormal{Fix}}

\usepackage[bookmarks=false,colorlinks=true,linkcolor=black,citecolor=black,filecolor=black,urlcolor=black]{hyperref} 
\usepackage[capitalise,nameinlink]{cleveref}

\theoremstyle{plain}
\newtheorem{theorem}{Theorem}[section]
\newtheorem*{theorem*}{Theorem}
\newtheorem{definition}[theorem]{Definition}
\newtheorem*{definition*}{Definition}

\newtheorem{lemma}[theorem]{Lemma}
\newtheorem{corollary}[theorem]{Corollary}
\newtheorem{proposition}[theorem]{Proposition}
\theoremstyle{definition}
\newtheorem{remark}[theorem]{Remark}
\newtheorem{notation}[theorem]{Notation}

\begin{document}

\title[A classification of exceptional components]{A classification of exceptional components in\\ group algebras over abelian number fields}

\author{Andreas B\"achle}
\address{Department of Mathematics, Vrije Universiteit Brussel,
Pleinlaan 2, 1050 Brussels, Belgium}
\email{abachle@vub.ac.be}
\author{Mauricio Caicedo}
\email{mcaicedo@vub.ac.be}
\author{Inneke Van Gelder}
\email{ivgelder@vub.ac.be}

\thanks{The research is supported by the Research Foundation Flanders (FWO - Vlaanderen), partially by the FWO project G.0157.12N}

\date{\today}



\begin{abstract}
  When considering the unit group of $\OO_F G$ ($\OO_F$ the ring of integers of an abelian number field $F$ and a finite group $G$) certain components in the Wedderburn decomposition of $FG$ cause problems for known generic constructions of units; these components are called exceptional. Exceptional components are divided into two types: type 1 are division rings, type 2 are $2 \times 2$-matrix rings. For exceptional components of type 1 we provide infinite classes of division rings by describing the seven cases of minimal groups (w.r.t.\ quotients) having those division rings in their Wedderburn decomposition over $F$. We also classify the exceptional components of type 2 appearing in group algebras of a finite group over number fields $F$ by describing all 58 finite groups $G$ having a faithful exceptional Wedderburn component of this type in $FG$. 
\end{abstract}

\maketitle

\section{Introduction}

Let $A$ be a semisimple algebra, finite dimensional over $\Q$. Consider the Wedderburn decomposition $A = \bigoplus_{i = 1}^k M_{n_i} (D_i)$, where the $D_i$ are division algebras. Let $\OO$ be a $\Z$-order in $A$, i.e.\ a subring of $A$, finitely generated as a $\Z$-module and $\Q\OO = A$. Furthermore, let $\OO_i$ be a $\Z$-order in $D_i$, hence $M_{n_i} (\OO_i)$ is a $\Z$-order in $M_{n_i} (D_i)$. It is well-known that the unit group $\U(\OO)$ is commensurable with $\bigoplus_{i = 1}^k \GL_{n_i} (\OO_i)$, i.e.\ they have a subgroup in common that is of finite index in both. Since $\GEN{\SL_{n_i}(\OO_i), \U(\mathcal{Z}(\OO_i))}$ has finite index in $\GL_{n_i}(\OO_i)$ and $\U(\mathcal{Z}(\OO_i))$ can be described via the Dirichlet Unit Theorem, in order to discover constructions of units which generate a subgroup of finite index of $\U(\OO)$, it is sufficient to solve the problem for $\SL_{n_i}(\OO_i)$. Let $Q$ be an ideal of a $\Z$-order $\OO$ in a rational division algebra and denote by $E_n(Q)$ the subgroup of $\SL_{n}(\OO)$ generated by all $Q$-elementary matrices, i.e.\ matrices having ones on the diagonal, one off-diagonal entry in $Q$ and zeros everywhere else. The results of \cite{BassMilnorSerre1967,Vaserstein1972,Vaserstein1973,Liehl1981,Venkataramana1994} are summarized in the following theorem. 

\begin{theorem*} \label{elementary}
Let $D$ be a finite dimensional rational division algebra and $\OO$ a $\Z$-order in $D$. Consider $M_n(D)$. Let $Q$ be any non-zero ideal of $\OO$.
If $n\geq 3$ then $[\SL_{n}(\OO):E_n(Q)]<\infty$. If $n = 2$ and $D$ is not $\Q$, a quadratic imaginary extension of $\Q$ or a totally definite quaternion algebra with center $\Q$, then $[\SL_{2}(\OO):E_2(Q)]<\infty$.
\end{theorem*}

The theorem is closely related to the famous Congruence Subgroup Problem (CSP) and the so-called congruence kernel. For $n = 1$ and $D$ non-commutative it is known that $\SL_{1}(\OO)$ is finite if and only if $D$ is a totally definite quaternion algebra \cite{2000Kleinert}. The answer to CSP in the case of non-commutative division algebras other than totally definite quaternion algebra is not known. The congruence kernel is known to be finite for all the cases covered by the theorem. For the excluded cases of the $\SL_2(\OO)$ the congruence kernel is infinite. For further reading we refer to \cite{Bass1964,Vaserstein1973,BakRehman1982,Sury2003,2010prasad}. This problem motivates the following definition.\pagebreak[5]

\begin{definition*}
 A simple finite dimensional rational algebra is called exceptional if it is of one of the following types:
\begin{enumerate}
 \item[type 1:] a non-commutative division algebra other than a totally definite quaternion algebra,
 \item[type 2:] a $2\times 2$-matrix ring over $\Q$, a quadratic imaginary extension of $\Q$ or over a totally definite quaternion algebra over $\Q$.
\end{enumerate}
\end{definition*}

From now on let $F$ be a number field and $\OO_F$ denote its ring of integers. Consider the group ring $\OO_F G$ of a finite group $G$. The above theorem is essential to construct with generic units a subgroup of finite index in the units $\U(\OO_F G)$. For instance, the so-called bicyclic units in $\OO_F G$ correspond to elementary matrices and hence to the $E_n(Q)$ in all Wedderburn components. This philosophy emerged in \cite{RiSe87,RS1989,RS1991,Sehgal1993,JespersLeal1993} to produce a generic subgroup of finite index in $\OO_F G$, provided $FG$ does not have exceptional Wedderburn components. Ritter and Sehgal showed that if $FG$ does not have exceptional components, then the center of $\U(\OO_FG)$ together with certain unipotent units generate a subgroup of finite index in $\U(\OO_F G)$ \cite[Theorem 1]{RiSe87}. Together with a result of Bass \cite{Bass1966}, Jespers and Leal proved that the Bass cyclic units together with the bicyclic units generate a subgroup of finite index in $\U(\Z[\xi] G)$ for groups $G$ which do not have a non-abelian epimorphic image that is fixed point free and such that $\Q(\xi)G$ does not have exceptional Wedderburn components \cite[Corollary 4.1]{JespersLeal1993}. Note that the notion of exceptional component in the above definition is a less restrictive version of the one used by Ritter, Sehgal, Jespers and Leal since the result of Venkataramana added new information in the case where $D$ is a non-commutative division algebra and $n=2$.

From the classification of Banieqbal \cite{Banieqbal} one can distill the exceptional simple algebras of type 2 which can occur in the Wedderburn decomposition of group rings of finite groups over number fields. He classified all finite subgroups of $2 \times 2$-matrices over division rings of characteristic zero. However, the description of the involved division rings is mostly implicit. Later Nebe \cite{NebeQuat} gave a list of finite subgroups of $\GL_n(D)$, spanning $M_n(D)$ over $\Q$, for $D$ a totally definite quaternion algebra with center of degree $d$ over $\Q$, such that $nd \leq 10$. In particular, for the case $n = 2$ and $d = 1$ she obtained an explicit description of the appearing division rings. Recently, the seven isomorphism types of exceptional components of type 2 which can occur in the Wedderburn decomposition of a group algebra, were provided with an elementary proof in \cite[Theorems 3.1 and 3.5]{2014EiseleKieferVanGelder}. For each of those seven $2 \times 2$-matrices over division rings $D$ there exists a left norm Euclidean (and therefore maximal) $\Z$-order $\OO$ in $D$. Thus one can use the Euclidean algorithm to find generators of $\SL_2(\OO)$. This extends the Jespers-Leal result to allow exceptional components of type 2 \cite[Method 4.2]{2014EiseleKieferVanGelder}. However, the former method relies on concrete isomorphisms between $M_2(D)$ and the simple component within the group algebra, which therefore does not allow generic constructions. Hence groups with exceptional components of type 2 still need more investigation.

For the unit groups of exceptional components of type 1 very little is known. Let $\OO$ be a $\Z$-order in a non-commutative division ring different from a totally definite quaternion algebra. Then $\SL_{1}(\OO)$ is infinite and to the best of our knowledge, there are no generic constructions of subgroups of finite index known. In 2000, Kleinert \cite{KleinertBook} gave a commendable survey on that topic. Up to that date no constructions were known for degree 3 division rings and also for degree 2 very little was known. Only recently, there was some progress made by \cite{Corrales2004,Kiefer2014} for degree 2 division rings. Braun, Coulangeon, Nebe and Sch{\"o}nnenbeck provided a generalization of Vorono\"i's algorithm to tackle the problem \cite{Nebe2014}, they give examples to work with division algebras of degree 2 and 3. Still one would like to have generic constructions such as the Pell and Gau\ss\ units introduced in \cite{Juriaans2009, Juriaans2013}, and moreover to have constructions of groups of units in $\OO_F G$ that contain a subgroup of finite index in $\U(\OO)$, when $\OO$ is a $\Z$-order in a division ring appearing in the Wedderburn decomposition of $FG$.

In this paper we first classify all exceptional components of type 2 occurring in the Wedderburn decomposition of group algebras of finite groups over arbitrary number fields. We do this by giving a full list of finite groups $G$, number fields $F$ and exceptional components $M_2(D)$ such that $M_2(D)$ is a faithful Wedderburn component of $FG$, cf.\ \cref{classification_matrices}. For producing this table, the classifications of either Banieqbal or Nebe can be used. The finite subgroups having $M_2(\Q)$ or $M_2(K)$ with $K$ a quadratic imaginary extension of $\Q$ are shown to be a subgroup of $\GL(2, 25)$ in \cite{2014EiseleKieferVanGelder}. 

Afterwards we deal with exceptional components of type 1. We classify $F$-critical groups, i.e.\ groups $G$ such that $FG$ has an exceptional component of type 1 in its Wedderburn decomposition, but no proper quotient has this property. In this way we obtain a minimal list of exceptional components of type 1 appearing in group algebras $FG$ for abelian number fields $F$ and $G$ finite. Regarding the scale of the difficulty of the problem, at least for the division rings in the list the unit groups of $\Z$-orders have to be studied. Note that any group $H$ such that $FH$ has a non-commutative division ring (not totally definite quaternion) in its Wedderburn decomposition has an epimorphic $F$-critical image $G$ such that if an exceptional component $D$ of type 1 appears as a faithful Wedderburn component of $FG$, then also $FH$ has $D$ as a simple component. Having an $F$-critical epimorphic image for a group implies that, up to now, there is no hope for a generic construction of units in $\OO_F G$. We give necessary and sufficient conditions for a finite group $G$ to be $F$-critical expressed in easy arithmetic formulas in terms of parameters of the group, and ramification indices and residue degrees of extensions of $F$, relying on the parameters of the group. For any abelian number field $F$ and any finite $F$-critical group $G$, we explicitly describe the division ring, which is an exceptional component of type 1 in $FG$, cf.\ \cref{classification_division}. In this part, we use the classification result of all finite subgroups of rational division algebras by Amitsur \cite{Amitsur}. 

All results presented in this work extend a result of \cite{2013Caicedo} where they gave an approximation to the problem of classifying for which finite groups $G$ the congruence kernel of $\U(\Z G)$ is finite, by using exceptional components. The results of \cite{Banieqbal,NebeQuat,2014EiseleKieferVanGelder} give more insight in the exceptional components of type 2, therefore it is natural to distinguish between exceptional components of type 1 and type 2, in contrast to what was done in \cite{2013Caicedo}. So our work yields (also for the case of rational group algebras) an extension of the above cited result.

\section{Preliminaries}\label{preliminaries}

We denote the set of positive integers by $\N$.
For integers $r$ and $m$, $\gcd(r,m)$ is the greatest common divisor of $r$ and $m$.
For integers $r,m$ and $p$ with $p$ prime and $\gcd(r,m)=1$ let
\begin{eqnarray*}
 v_p(m) &=& \mbox{maximum non-negative integer } k \mbox{ such that } p^k \mbox{ divides } m; \\
 o_m(r) &=& \mbox{multiplicative order of } r \mbox{ modulo } m; \\
 \zeta_m &=& \mbox{complex primitive } m\mbox{-th root of unity}.
\end{eqnarray*}

Given a field $F$ and two elements $a,b\in F$ we define 
	the quaternion algebra $\quat{a, b}{F}$ as follows:
	\begin{equation*}
		\quat{a,b}{F} = \frac{F\langle i, j\rangle}{(i^2=a,\ j^2=b,\ ij=-ji)}.
	\end{equation*}
It is a well known result that $\quat{a, b}{F}$ is a division ring if and only if the equation $ax^2+by^2=z^2$ does not have a non-zero solution $(x,y,z)$ in $F^3$  \cite[Proposition 1.6]{Pierce1982}.

A group $G$ will always be assumed to be finite.
For subgroups $G$ and $H$ of a common group, we denote $[G,H]$ to be the commutator of $G$ and $H$, i.e. the subgroup generated by all elements $g^{-1}h^{-1}gh$ with $g\in G$ and $h\in H$. By $\GEN{a}_m$ we denote the cyclic group of order $m$ generated by an element $a$. $\GEN{a}_m \rtimes_k \GEN{b}_n$ will always denote a semidirect product of $\GEN{b}_n$ acting on $\GEN{a}_m$ with kernel of order $k$, this means that if $b^{-1}ab = a^r$, then $o_m(r) = \frac{n}{k}$.
Let $F$ be a number field. If $\alpha \in FG$ and $g \in G$ then $\alpha^g = g\inv \alpha g$ and $\Cen_G(\alpha)$ denotes the centralizer of $\alpha$ in $G$. The notation $H\leq G$ (respectively, $H \unlhd G$) means that $H$ is a subgroup (resp., normal subgroup) of $G$. If $H\leq G$ then $N_G(H)$ denotes the normalizer of $H$ in $G$ and we set $\suma{H} = \frac{1}{|H|}\sum_{h\in H}h$. If $g\in G$ then $\suma{g}=\suma{\GEN{g}}$.

If $R$ is a unital associative ring and $G$ is a group then $R*^{\alpha}_{\tau} G$ denotes a crossed product with action $\alpha:G\rightarrow \Aut(R)$ and twisting (a two-cocycle) $\tau:G\times G \rightarrow \U(R)$ (see for example \cite{Passman1989,Reiner1975}), i.e. $R*^{\alpha}_{\tau} G$ is the associative ring $\bigoplus_{g\in G} R u_g$ with multiplication given by the following rules: $u_g a = \alpha_g(a) u_g$ and $u_g u_h = \tau(g,h) u_{gh}$, for $a\in R$ and $g,h\in G$. 
Let $F$ be a field and $\zeta$ a root of unity in an extension of $F$. If $\Gal(F(\zeta)/F)=\GEN{\sigma_n}$ is cyclic of order $n$, we can consider the cyclic cyclotomic algebra $F(\zeta)*^{\alpha}_{\tau} \Gal(F(\zeta)/F)$ where $\alpha$ is the natural action on $F(\zeta)$, i.e. $\alpha_{\sigma_n^m}=\sigma_n^m$. Also $\tau(g,h)$ is a root of unity for every $g$ and $h$ in $\Gal(F(\zeta)/F)$ and $\tau$ is completely determined by $u_{\sigma_n}^n=\zeta^c$. We denote this cyclic cyclotomic algebra by $(F(\zeta)/F,\sigma_n,\zeta^c)$. Furthermore, if $F$ is a field and $G$ is a group of automorphisms of $F$, we denote by $\Fix(G)$, or $\Fix(\sigma)$, if $G$ is cyclic generated by $\sigma$, the fixed subfield of $F$ under $G$. Note that if $p$ is a prime number not dividing $r$, then $\Fix(\Q(\zeta_p)\rightarrow \Q(\zeta_p):\zeta_p\mapsto \zeta_p^r)=\Q(\zeta_p+\zeta_p^r+...+\zeta_p^{r^{o_p(r)-1}})$, see {\cite[Example 14.5.2]{Dummit2004}}.

In 2004, Olivieri, del R\'io and Sim\'on \cite{Olivieri2004} showed a method to describe the primitive central idempotents of $\Q G$ for finite strongly monomial groups (including abelian-by-supersolvable groups). We recall this method.

If $K\unlhd H\leq G$ and $K\neq H$ then let $$\varepsilon(H,K)=\prod (\suma{K}-\suma{M})=\suma{K}\prod (1-\suma{M}),$$ where $M$ runs through the set of all minimal normal subgroups of $H$ containing $K$ properly. We extend this notation by setting $\varepsilon(H,H)=\suma{H}$. Let $e(G,H,K)$ be the sum of the distinct $G$-conjugates of $\varepsilon(H,K)$.

A strong Shoda pair of $G$ is a pair $(H,K)$ of subgroups of $G$ satisfying $K\unlhd H\unlhd N_G(K)$, $H/K$ is cyclic and a maximal abelian subgroup of $N_G(K)/K$ and the different $G$-conjugates of $\varepsilon(H,K)$ are orthogonal. In this case $\Cen_G(\varepsilon(H,K))=N_G(K)$. 

Let $\chi$ be an irreducible (complex) character of $G$. One says that $\chi$ is strongly monomial if there is a strong Shoda pair $(H,K)$ of $G$ and a linear character $\theta$ of $H$ with kernel $K$ such that $\chi=\theta^G$, the induced character of $G$. The group $G$ is strongly monomial if every irreducible character of $G$ is strongly monomial. For finite strongly monomial groups all primitive central idempotents of $\Q G$ are of the form $e(G,H,K)$ with $(H,K)$ a strong Shoda pair of $G$. 

%

More information was obtained on the strong Shoda pairs needed to describe the primitive central idempotents of the rational group algebra of a finite metabelian group.

\begin{theorem}[{\cite[Theorem 4.7]{Olivieri2004}}]\label{SSPmetabelian}
Let $G$ be a finite metabelian group and let $A$ be a maximal abelian subgroup of $G$ containing the commutator subgroup $G'$. The primitive central idempotents of $\Q G$ are the elements of the form $e(G,H,K)$, where $(H,K)$ is a pair of subgroups of $G$ satisfying the following conditions:
\begin{enumerate}
\item \label{metabelian1}$H$ is a maximal element in the set $\{B\leq G \mid A\leq B \mbox{ and } B'\leq K\leq B\}$;
\item \label{metabelian2}$H/K$ is cyclic.
\end{enumerate}
\end{theorem}

\cref{SSPmetabelian} 
allows one to easily compute the primitive central idempotents 
of the rational group algebra of a finite metacyclic group. Every finite metacyclic group $G$ has a presentation of the form $$G=\GEN{a,b\mid a^m=1,\ b^n=a^t,\ a^b=a^r},$$ where $m,n,t,r$ are integers satisfying the conditions $r^n \equiv 1 \mod m$ and $m\mid t(r-1)$. Let $u=o_m(r)$, then $u\mid n$.
For every $d\mid u$, let $G_d=\langle a,b^d\rangle$. 
\begin{lemma}\label{SSPmetacyclic}
 With notations as above, the primitive central idempotents of $\Q G$ are the elements of the form $e(G,G_d,K)$ where $d$ is a divisor of $u$ and $K$ is a subgroup of $G_d$ satisfying the following conditions:
\begin{enumerate}
\item $d=\min\{x\mid u : a^{r^x-1}\in K\}$,
\item $G_d/K$ is cyclic.
\end{enumerate}
\end{lemma}

Olivieri, del R\'io and Sim\'on also provided information on the Wedderburn decomposition of $\Q G$. We will immediately state this in the more general context of number fields as in \cite{2014OlteanuGelder}.

Let $F$ be any number field. If $G=\GEN{g}$ is cyclic of order $k$, then the irreducible (complex) characters are all linear and are defined by the image of a generator of $G$. Therefore the set ${G^*=\Irr(G)}$ of irreducible characters of $G$ is a group and the map $\phi:\Z/k\Z\rightarrow G^*$ given by $\phi(m)(g)=\zeta_k^m$ is a group homomorphism. The generators of $G^*$ are precisely the faithful representations of $G$. Let $\mathcal{C}_F(G)$ denote the set of orbits of the faithful characters of $G$ under the action of $\Gal(F(\zeta_k)/F)$. 
Each automorphism $\sigma\in\Gal(F(\zeta_k)/F)$ is completely determined by its action on $\zeta_k$, and is given by $\sigma(\zeta_k)=\zeta_k^t$, where $t$ is an integer uniquely determined modulo $k$. In this way, one gets the following morphisms 
\begin{center}
\begin{tikzpicture}
 \matrix (m) [matrix of math nodes,row sep=3em,column sep=4em,minimum width=2em]
 {
   \Gal(F(\zeta_k)/F) & \Gal(\Q(\zeta_k)/\Q) \\
   I_k(F) & \U(\Z/k\Z) \\};
 \path[-stealth]
  (m-1-1) edge node [left] {$\iso$} (m-2-1)
      edge[draw=none] node [sloped, auto=false, allow upside down] {$\hookrightarrow$} (m-1-2)
  (m-2-1.east|-m-2-2) edge[draw=none] node [sloped, auto=false, allow upside down] {$\hookrightarrow$} (m-2-2)
  (m-1-2) edge node [right] {$\iso$} (m-2-2);
\end{tikzpicture}
\end{center}
where we denote the image of $\Gal(F(\zeta_k)/F)$ in $ \U(\Z/k\Z)$ by $I_k(F)$.
In this setting, the sets in $\mathcal{C}_F(G)$ corresponds one-to-one to the orbits under the action of $I_k(F)$ on $\U(\Z/k\Z)$ by multiplication.


Let $N\unlhd G$ be such that $G/N$ is cyclic of order $k$ and $C\in\mathcal{C}_F(G/N)$. If $\chi\in C$ and $\tr=\tr_{F(\zeta_k)/F}$ denotes the field trace of the Galois extension $F(\zeta_k)/F$, then we set 
$$\varepsilon_C(G,N)=\frac{1}{|G|}\sum_{g\in G} \tr(\chi(gN))g\inv=\frac{1}{|G|}\sum_{g\in G}\sum_{\psi\in C}\psi(gN)g\inv.$$ 

Let $K\unlhd H\leq G$ such that $H/K$ is cyclic and $C\in\mathcal{C}_F(H/K)$. Then $e_C(G,H,K)$ denotes the sum of the different $G$-conjugates of $\varepsilon_C(H,K)$.


Now let $K\unlhd H\unlhd N_G(K)$ be such that $H/K$ is cyclic of order $k$. 
Fix a generator $hK$ of $H/K$. We define $E_F(G,H/K)$ to be the set $\{g\in N_G(K)\mid g\inv hgK=h^iK \mbox{ for some } i\in I_k(F)\}$.

In the case of $F = \Q$ we have $I_{[H:K]}(\Q) = \U(\Z/[H:K]\Z)$, $|\mathcal{C}_F(H/K)| = 1$ and $E_F(G,H/K) = N_G(K)$, therefore we omit the $C$ in the notation $e_C(G,H,K)$ and denote it by $e(G,H,K)$.

We can now describe the simple components of $FG$.
\begin{theorem}[{\cite[Theorem 3.2]{2014OlteanuGelder}}]\label{main}
Let $G$ be a finite group and $F$ be a number field.
\begin{enumerate}
 \item Let $(H,K)$ be a strong Shoda pair of $G$ and $C\in\mathcal{C}_F(H/K)$. Let $[H:K]=k$ and $E=E_F(G,H/K)$. Then $e_C(G,H,K)$ is a primitive central idempotent of $FG$ and $$FGe_C(G,H,K)\iso M_{[G:E]}\left(F\left(\zeta_{k}\right)*_{\tau}^{\sigma}E/H\right),$$ where $\sigma$ and $\tau$ are defined as follows. Let $\phi:E/H\rightarrow E/K$ be a left inverse of the projection $E/K\rightarrow E/H$ and $yK$ be a generator of $H/K$. Then
\begin{eqnarray*}
\sigma_{gH}(\zeta_k) &=& \zeta_k^i, \mbox{ if } yK^{\phi(gH)}=y^iK ,\\
\tau(gH,g'H) &=& \zeta_k^j, \mbox{ if } \phi(gg'H)^{-1}\phi(gH)\phi(g'H)=y^jK,
\end{eqnarray*}
for $gH,g'H\in E/H$ and integers $i$ and $j$. The center of $FGe_C(G,H,K)$ is isomorphic to $\Fix(E/H)$, the fixed field of $F\left(\zeta_{k}\right)$ under the action of $E/H$.
\item Let $X$ be a set of strong Shoda pairs of $G$. If every primitive central idempotent of $\Q G$ is of the form $e(G,H,K)$ for $(H,K)\in X$, then every primitive central idempotent of $FG$ is of the form $e_C(G,H,K)$ for $(H,K)\in X$ and $C\in\mathcal{C}_F(H/K)$.
\end{enumerate}
\end{theorem}

%

\begin{definition} Let $F$ be a number field and $G$ a finite group. Let $B$ be an exceptional simple algebra. If $FG$ has $B$ as a component in its Wedderburn decomposition, we say that $B$ is an \emph{exceptional component} of $FG$. We call a Wedderburn component $B$ of $FG$ \emph{faithful} if $G$ is faithfully embedded in $B$ via the Wedderburn isomorphism. 
\end{definition}

\begin{theorem}[{\cite[Theorem 3.1]{2014EiseleKieferVanGelder}}]\label{florian1}
If $G$ is a finite subgroup of $\GL_2(F)$, for $F$ a quadratic imaginary extension of the rationals, such that $G$ spans $M_2(F)$ over $F$, then $G$ is solvable, ${|G|=2^a 3^b}$ for $a,b\in$ $\N$ and $F$ is one of the following fields:
\begin{enumerate}
\item \label{main1type2} $\Q(\sqrt{-1})$, 
\item \label{main1type3} $\Q(\sqrt{-2})$ or 
\item \label{main1type4 }$\Q(\sqrt{-3})$.
\end{enumerate}

Furthermore, elements of a finite subgroup $G$ of $\GL_2(\Q)$ can only have prime power orders 1, 2, 3 and 4.
\end{theorem}

\begin{proposition}[{\cite[Proposition 3.3]{2014EiseleKieferVanGelder}}]\label{florian2}
Let $F\in\left\{\Q, \Q(\sqrt{-1}), \Q(\sqrt{-2}), \Q(\sqrt{-3})\right\}$, and let $G\leq \GL_2(F)$ be a finite group. 
Then $G$ can be embedded in the finite group $\GL(2,25)$.
\end{proposition}

\begin{theorem}[{\cite[Theorem 3.5]{2014EiseleKieferVanGelder}}]\label{florian3}
Let $G$ be a finite group. If $G$ is a subgroup of $\GL_2(D)$ for $D$ a totally definite quaternion algebra with center $\Q$, such that $G$ spans $M_2(D)$ over $\Q$, then $D$ is one of the following algebras:
\begin{enumerate}
\item \label{main2type1}$\quat{-1, -1}{\Q}$,
\item \label{main2type2}$\quat{-1, -3}{\Q}$ or 
\item\label{main2type3}$\quat{-2, -5}{\Q}$.
\end{enumerate}
\end{theorem}

The previous theorem can also be derived from \cite[Theorems 6.1 and 12.1]{NebeQuat}.

\begin{lemma}\label{subgroup}
 Let $G$ be a finite group, $F$ be a number field and $\rho$ an irreducible $F$-representation of $G$. Let $e$ be the primitive central idempotent associated to $\rho$ and $K$ be the kernel of $\rho$. Then the group $Ge$ is faithfully embedded in $(FG)e$ and its $F$-span equals $(FG)e$. In particular, $Ge$ is isomorphic to a subgroup of $\U((FG)e)$.
\end{lemma}

\begin{proof}
 Consider the irreducible representation $\rho \colon G \to \U(FGe)$. This induces a faithful representation $\overline{\rho} \colon G/K \to \U(FGe)$ and $G/K \simeq \rho(G) = Ge$. Since $FGe$ is the $F$-span of $\rho(G) = \overline{\rho}(G/K)$ and $FGe$ is simple, $FGe$ is also isomorphic to a Wedderburn component of $F (G/K)$.
\end{proof}


In this paper we are interested to know all exceptional components of a group algebra $FG$ of a finite group over a number field. Because of the previous lemma, it is sufficient to study finite subgroups of $\U(B)$, which span $B$ over $F$, for exceptional simple algebras $B$. By \cref{florian1,florian3}, the exceptional components of type 2 restrict to the following seven isomorphism types: $M_2(\Q), M_2(\Q(\sqrt{-1})), M_2(\Q(\sqrt{-2})), M_2(\Q(\sqrt{-3})), M_2\quat{-1, -1}{\Q}, M_2\quat{-1, -3}{\Q}, M_2\quat{-2, -5}{\Q}$. If $B$ is a $2\times 2$-matrix ring over a field, then $G$ is a subgroup of $\GL(2,25)$, by \cref{florian2}. When $B$ is a $2 \times 2$-matrix ring over a totally definite quaternion algebra, then $G$ appears in the classifications of Banieqbal and Nebe \cite{Banieqbal,NebeQuat}. 

When $B$ is a division ring, one can use Amitsur's classification \cite{Amitsur}. This classification splits into 2 parts: a list of some Z-groups and a list of non-Z-groups. Remind that Z-groups are groups with all Sylow $p$-subgroups cyclic. We use notations from \cite[Theorems 2.1.4, 2.1.5]{1986Shirvani} and \cite[Theorem 2.2]{2013Caicedo}.

\begin{theorem}[Amitsur]\label{SubgruposAD}
\leavevmode
\begin{enumerate}
 \item[(Z)] The Z-groups which are finite subgroups of division rings are:
	\begin{enumerate}
	 \item the finite cyclic groups,
	 \item $C_m \rtimes_2 C_4$ with $m$ odd and $C_4$ acting by inversion on $C_m$ and
  \item $C_m\rtimes_k C_n$ with $\gcd(m,n)=1$ and, using the following notation
  \begin{eqnarray*}
  P_p&=& \text{ Sylow } p\text{-subgroup of } C_m,\\
  Q_p&=& \text{ Sylow } p\text{-subgroup of } C_n,\\
  X_p&=&\{q\mid n : q \text{ prime and } [P_p,Q_q]\ne 1\},\\
  R_p&=&\prod_{q\in X_p} Q_q;
  \end{eqnarray*}
  we assume $C_n=\prod_{p\mid m} R_p$ and the following properties hold for every prime $p\mid m$ and $q\in X_p$:
    \begin{enumerate}
    \item $q\cdot o_{q^{v_q(k)}}(p) \nmid o_{\frac{mn}{\vert P_p\vert \; \vert R_p\vert }}(p)$,
    \item if $q$ is odd or $p\equiv 1 \mod 4$ then $v_{q}(p-1) \le v_q(k)$ and
    \item if $q=2$ and $p\equiv -1 \mod 4$ then $v_2(k)$ is either $1$ or greater than $v_2(p+1)$.
	  \end{enumerate}
  \end{enumerate}
 \item[(NZ)] The finite subgroups of division rings which are not Z-groups are:
	\begin{enumerate}
	\item $\mathcal{O}^{*}=\GEN{s,t|(st)^2=s^3=t^4}$ (binary octahedral group),
	\item $\SL(2,5)$,
	\item $\SL(2,3)\times M$, with $M$ a group in (Z) of order coprime to 6 and $o_{|M|}(2)$ odd,
	\item $Q_{4k}$ with $k$ even,
	\item $Q_8\times M$ with $M$ a group in (Z) of odd order such that $o_{|M|}(2)$ is odd.
	\end{enumerate}
\end{enumerate}
\end{theorem}

\begin{remark}\label{remarkZ}
Assume that $G=C_m\rtimes_k C_n$ is a group satisfying the hypothesis in (Z)(c) in \cref{SubgruposAD}. Let $p$ be a prime divisor of $m$ such that $C_n$ acts non-trivially on $P_p$. Let $q_1,...,q_h$ be the prime divisors $q$ of $n$ such that $Q_q$ acts non-trivially on $P_p$, so $R_p=Q_{q_1}\cdots Q_{q_h}$. Let $k_p$ be the order of the kernel of the action of $R_p$ on $P_p$. Then $P_p\rtimes_{k_p} R_p$ is a direct factor of $G$. From the conditions on (Z)(c) it follows that for every prime divisor $q_i$ of $|R_p|$, we have $v_{q_i}(p-1)\le v_{q_i}(k)$. Since $C_n=\prod_{p\mid m} R_p$ is a direct product, all $X_p$ are mutually disjoint, $k=\prod_{p\mid m}k_p$ and $k_p$ is only divisible by $q$ if $q\in X_p$. Therefore $v_{q_i}(k)=v_{q_i}(k_p)$.

Let $p^{\gamma}$ be the order of $P_p$, and let the action of $R_p$ on $P_p$ be defined by $\sigma: R_p \rightarrow \Aut(P_p)$, then $R_p/\text{Ker}(\sigma) \iso \text{Im}(\sigma)$, and hence $\frac{|R_p|}{k_p}$ divides $|\Aut(P_p)|=p^{\gamma -1}(p-1)$. Thus $\frac{|R_p|}{k_p}$ divides $p-1$. 
 
On the other hand, for all $1\le i\le h$, let $q_i^{\beta_i}$ and $q_i^{\alpha_i}$ be the order of $Q_{q_i}$ and the order of the kernel of the action of $Q_{q_i}$ on $P_p$, respectively. Hence $|R_p|=q_1^{\beta_1}\cdots q_h^{\beta_h}$ and $k_p=q_1^{\alpha_1}\cdots q_h^{\alpha_h}$. Given a fixed $i$, as $P_p\rtimes_{q_i^{\alpha_i}} Q_{q_i}$ is not cyclic, $\beta_i\ne \alpha_i$, so that $q_i$ divides $\frac{|R_p|}{k_p}$. It now follows that 
\begin{equation}\label{Zinequality}
1\le v_{q_i}\left(\frac{|R_p|}{k_p}\right)\le v_{q_i}(p-1)\le v_{q_i}(k_p).
\end{equation}

Note that we can deduce from \eqref{Zinequality} that every prime divisor of $|R_p|$ is a prime divisor of $\frac{|R_p|}{k_p}$ and of $k_p$. In particular, this means that any prime divisor of $n$ divides both $k$ and $\frac{n}{k}$. Finally, by \eqref{Zinequality} every $\alpha_i$ is not zero, this implies that none $Q_{q_i}$ acts faithfully on $P_p$, and so does $R_p$. So $k_p\neq 1$, and in particular $k\neq 1$. 
\end{remark}

Let $F$ be a number field and $A$ be a central simple $F$-algebra. The \emph{degree} of $A$ is defined as $\sqrt{\dim_F A}$. Also, $A=M_n(D)$ for some division algebra $D$. The \emph{Schur index} of $A$, denoted by $\ind(A)$, is the degree of $D$.

A prime of $F$ is an equivalence class of valuations of $F$. There are the infinite primes of $F$ arising from embeddings of $F$ into $\C$ and the finite primes of $F$, arising from discrete $\mathfrak{p}$-adic valuations of $F$, with $\mathfrak{p}$ ranging over the distinct prime ideals in the ring of integers of $F$. We denote by $F_{\mathfrak{p}}$ the $\mathfrak{p}$-adic completion of $F$. Then the \emph{local Schur index} of $A$ at $\mathfrak{p}$ is defined as $m_{\mathfrak{p}}(A)=\ind(F_{\mathfrak{p}}\otimes_F A)$.

Let $L$ be a number field containing $F$, $R$ be the ring of integers of $F$, $S$ the ring of integers of $L$ and $\mathfrak{P}$ a finite prime of $S$. Then $\mathfrak{P}\cap R=\mathfrak{p}$, a prime in $R$. We say that $\mathfrak{P}$ lies over $\mathfrak{p}$. The primes lying over a given prime $\mathfrak{p}$ are the $\mathfrak{P}_1,...,\mathfrak{P}_s$ which occur in the prime decomposition of $\mathfrak{p}S=\mathfrak{P}_1^{e_1}\cdots\mathfrak{P}_s^{e_s}$. We also say that $\mathfrak{P}_i$ divides $\mathfrak{p}$. The exponents $e_i$ are called the ramification indices and are denoted by $e_{\mathfrak{p}}^{\mathfrak{P}_i}(L/F)$. If $\mathfrak{P}$ is a prime lying over $\mathfrak{p}$ in the extension $L/F$, then $R/\mathfrak{p}$ can be viewed as a subfield of $S/\mathfrak{P}$ and we call the degree of $S/\mathfrak{P}$ over $R/\mathfrak{p}$ the residue degree of $\mathfrak{P}$ over $\mathfrak{p}$ and denote it by $f_{\mathfrak{p}}^{\mathfrak{P}}(L/F)$.
If $L/F$ is a normal extension and $\mathfrak{p}$ is a prime of $R$, then the Galois group $\Gal(L/F)$ permutes transitively the primes lying over $\mathfrak{p}$. Therefore both the ramification index and residue degree are independent on the choice of prime lying over $\mathfrak{p}$ and we write $e_{\mathfrak{p}}(L/F)$ and $f_{\mathfrak{p}}(L/F)$.

The following theorem is a well-known consequence of the Brauer-Hasse-Noether-Albert Theorem.
\begin{theorem}\label{lcm}
 Let $A$ be a central simple $F$-algebra having local indices $\{m_{\mathfrak{p}}(A)\}$. Then we have $\ind(A)=\lcm\{m_{\mathfrak{p}}(A)\}$.
\end{theorem}

The following proposition is due to Olteanu.
\begin{proposition}[{\cite[Proposition 2.13]{2009Olteanu}}]\label{olteanu}
Let $A=(F(\zeta_n)/F,\sigma,\zeta_m)$ for $F$ a number field. If $\mathfrak{p}$ is a prime of $F$, then $m_{\mathfrak{p}}(A)$ divides $m$. If $m_{\mathfrak{p}}(A)\neq 1$ for a finite prime $\mathfrak{p}$, then $\mathfrak{p}$ divides $n$ in $F$. 
 \end{proposition}

We also have the following reciprocity rules.

\begin{proposition}[{\cite[Proposition 1.6.2]{1963Weiss}}]\label{reciprocity}
 Let $K/L/F$ be normal extensions of number fields. Let $\mathfrak{p}$ be a prime in $F$ and $\mathfrak{P}$ a prime in $L$ lying over $\mathfrak{p}$. Then
\begin{enumerate}
 \item $e_{\mathfrak{p}}(K/F)=e_{\mathfrak{P}}(K/L) e_{\mathfrak{p}}(L/F)$;
 \item $f_{\mathfrak{p}}(K/F)=f_{\mathfrak{P}}(K/L) f_{\mathfrak{p}}(L/F)$.
\end{enumerate}
\end{proposition}

Because of these rules, the notations $e_p(L/F)$ and $f_p(L/F)$ are unambiguously for normal extensions $L/F/\Q$ and rational primes $p$.

Similar as for ramification index and residue degree, we can also restrict the computation of local Schur indices to rational primes when working over abelian number fields.

\begin{theorem}[\cite{Benard}]
Let $F$ be an abelian number field and $A$ a central simple $F$-algebra.
As $\mathfrak{p}$ runs over the set of primes lying over the same (infinite or finite) rational prime $p$, the local indices $m_{\mathfrak{p}}(A)$ are all equal to the same positive integer, which we call the $p$-local index of $A$ and denote by $m_p(A)$.
\end{theorem}

The following proposition is folklore. We provide a proof for reasons of completeness.

\begin{proposition}\label{splitting}
 Let $A = \quat{a,b}{K}$ with $K$ a finite Galois extension of $\Q$, $a, b \in K$ totally negative and let $F$ be a finite Galois extension of $\Q$ containing $K$. Then $F \otimes_K A$ is a division algebra if and only if $F$ is a totally real number field or 
 there exists a prime number $p$ such that $m_{p}(A) \not= 1$ and both $e_{p}(F/K)$ and $f_{p}(F/K)$ are odd.
\end{proposition} 

\begin{proof} Let $B = F \otimes_K A$. Note that $B$ is a division algebra if and only if $m_p(B) \neq 1$ for some prime number $p$.

Assume that $F$ is totally real, then there exists a real embedding $\sigma$ of $F$ and for the completion of this embedding we find $\mathbb{R} \otimes_{F} B \iso \quat{\sigma(a), \sigma(b)}{\mathbb{R}}$, hence $m_{\infty}(B) = 2$ and $B$ is a division algebra. 

Assume that $F$ is totally imaginary, then for all embeddings of $F$ all local Schur indices for infinite primes are $1$, since the completion of any embedding of $F$ is $\mathbb{C}$ and splits $B$.

Fix a finite prime $p\in \Q$. If $m_p(A) = 1$, then clearly $m_p(B) = 1$. If $m_p(A) \not= 1$, then it is $2$ and due to \cite[Corollary 31.10]{Reiner1975} or \cite[Satz 2 on p.\ 118]{Deuring}, $m_p(B) = 1$ if and only if $m_p(A) \mid [F_p: K_p]$.
Hence the result follows since $[F_p: K_p]=e_{p}(F/K) f_{p}(F/K)$.
\end{proof}

\cref{splitting} is a generalization of the following result.
\begin{corollary}[{\cite[Theorem 2.1.9]{1986Shirvani}}]\label{Shirvani-2.1.9}
 For $F$ a finite Galois extension of $\Q$, $F\otimes_{\Q}\quat{-1,-1}{\Q}$ is a division algebra if and only if $F$ is a totally real field or both $e_{2}(F/\Q)$ and $f_{2}(F/\Q)$ are odd.
\end{corollary}


Lately, Herman gave a nice review on the computation of local Schur indices for cyclic cyclotomic algebras in \cite{2014Herman}.

\begin{lemma}[{\cite[Lemma 2]{2014Herman}}]\label{hermaninfty}
 A cyclic cyclotomic algebra $A=(F(\zeta_n)/F,\sigma_b,\zeta_n^c)$ over an abelian number field $F$ has local index 2 at an infinite prime if and only if $F\subseteq \R$, $n>2$ and $\zeta_n^c=-1$.
\end{lemma}

The following lemma follows from Janusz. 
\begin{lemma}[{\cite[Lemma 3.1]{1975Janusz}}]\label{hermanodd}
Let $p$ be an odd rational prime and $F$ an abelian number field. Let $e=e_p(F(\zeta_n)/F)$ and $f=f_p(F/\Q)$. 
Then 
\begin{eqnarray*}
 m_p(F(\zeta_n)/F,\sigma_b,\zeta_n^c) &=& \min\{l\in \N \mid \zeta_n^{cl}\in \GEN{\zeta_{p^{f}-1}^e}\}\\
&=&\min\left\{l\in \N \left| \frac{p^f-1}{\gcd(p^f-1,e)} \equiv 0 \mod \frac{n}{\gcd(n,cl)}\right.\right\}.
\end{eqnarray*}
\end{lemma}


The computation of local Schur indices is made available in the \texttt{GAP}-package \texttt{wedderga} by Herman \cite{2014Herman,Wedderga}. 


\section{Group algebras with exceptional components of type 2}\label{section:type2}

We give a full list of finite groups $G$ and number fields $F$ having faithful exceptional components of type 2 in $FG$. Employing \cref{subgroup} one can deduce from this list all exceptional components of type 2 within group algebras over number fields.

\begin{theorem}\label{classification_matrices}
Let $F$ be a number field, $G$ be a finite group and $B$ a simple exceptional algebra of type 2.
Then $B$ is a faithful Wedderburn component of $FG$ if and only if $G$, $F$, $B$ is a row listed in \cref{exceptional_table}.
\end{theorem}
\begin{proof}
Let $B$ be a simple exceptional algebra of type 2 and assume that $B$ is a faithful Wedderburn component of $FG$, then by \cref{subgroup}, $G$ is a subgroup of $\U(B)$ and $B$ is isomorphic to an algebra stated in \cref{florian1,florian3}.

The subgroups of $M_2(\Q),M_2(\Q(\sqrt{-1})),M_2(\Q(\sqrt{-2})),M_2(\Q(\sqrt{-3}))$ are embedded in $\GL(2,25)$ by \cref{florian2}.
The maximal finite subgroups of $2\times 2$-matrices over totally definite quaternion algebras with center $\Q$ were classified in
\cite{NebeQuat}. They can be accessed in \texttt{Magma} \cite{magma} by calling \texttt{QuaternionicMatrixGroupDatabase()}.

It is also clear that when $FG$ has a Wedderburn component $B$ then $F$ is contained in the center of $B$, which restricts the possibilities of $F$ for $G$ and $B$ fixed.

Additionally, using the \texttt{GAP}-package \texttt{wedderga}, one can compute a finite list of groups $G$ that have $B$ as a faithful component over $F$. We mainly use the function \texttt{WedderburnDecomposition\-WithDivAlgParts} which returns the size of the matrices, the centers and the local indices of all Wedderburn components of a group algebra and allows us to compare the Wedderburn components to the possibilities of $B$ above. This is possible, since $F$ is a number field and the isomorphism type of division algebras is determined by its list of local Hasse invariants at all primes of $F$.
For quaternion algebras the local Hasse invariants are uniquely determined by the local Schur indices. \end{proof}

\begin{notation}[in \cref{exceptional_table}]
We use the \texttt{GAP} notation for the group structure.
Direct products are denoted by $A\times B$ and semidirect products by $A\rtimes B$.
For a non-split extension of $A$ by $B$, we write $A.B$. If an exceptional component appears several times in $FG$, this multiplicity is indicated in the last column.
\end{notation}

{\small 
\begin{longtable}{@{}llp{2cm}l@{}} \caption{List of all groups having a faithful exceptional component of type 2} \label{exceptional_table} \\ \toprule[1.5pt]
\textsc{SmallGroup} ID & Structure & $F$ & $B$  \\ \midrule 
\endfirsthead \toprule[1.5pt] \textsc{SmallGroup} ID & Structure & $F$ & $B$  \\ \midrule 
\endhead \hline \multicolumn{4}{c}{continued}\\ \midrule[1.5pt]\endfoot\bottomrule[1.5pt]\endlastfoot
{[}6, 1{]} & $ S_3 $ & $ \mathbb Q $& $1 \times $ $ M_2\left( \mathbb Q \right)$ \\
{[}6, 1{]} & $ S_3 $ & $ \mathbb Q (\sqrt{-1}) $& $1 \times $ $ M_2\left( \mathbb Q (\sqrt{-1}) \right)$ \\
{[}6, 1{]} & $ S_3 $ & $ \mathbb Q (\sqrt{-2}) $& $1 \times $ $ M_2\left( \mathbb Q (\sqrt{-2}) \right)$ \\
{[}6, 1{]} & $ S_3 $ & $ \mathbb Q (\sqrt{-3}) $& $1 \times $ $ M_2\left( \mathbb Q (\sqrt{-3}) \right)$ \\
{[}8, 3{]} & $ D_8 $ & $ \mathbb Q $& $1 \times $ $ M_2\left( \mathbb Q \right)$ \\
{[}8, 3{]} & $ D_8 $ & $ \mathbb Q (\sqrt{-1}) $& $1 \times $ $ M_2\left( \mathbb Q (\sqrt{-1}) \right)$ \\
{[}8, 3{]} & $ D_8 $ & $ \mathbb Q (\sqrt{-2}) $& $1 \times $ $ M_2\left( \mathbb Q (\sqrt{-2}) \right)$ \\
{[}8, 3{]} & $ D_8 $ & $ \mathbb Q (\sqrt{-3}) $& $1 \times $ $ M_2\left( \mathbb Q (\sqrt{-3}) \right)$ \\
{[}8, 4{]} & $ Q_8 $ & $ \mathbb Q (\sqrt{-1}) $& $1 \times $ $ M_2\left( \mathbb Q (\sqrt{-1}) \right)$ \\
{[}8, 4{]} & $ Q_8 $ & $ \mathbb Q (\sqrt{-2}) $& $1 \times $ $ M_2\left( \mathbb Q (\sqrt{-2}) \right)$ \\
{[}8, 4{]} & $ Q_8 $ & $ \mathbb Q (\sqrt{-3}) $& $1 \times $ $ M_2\left( \mathbb Q (\sqrt{-3}) \right)$ \\
{[}12, 1{]} & $ C_3  \rtimes   C_4$ & $ \mathbb Q (\sqrt{-1}) $& $1 \times $ $ M_2\left( \mathbb Q (\sqrt{-1}) \right)$ \\
{[}12, 1{]} & $ C_3  \rtimes   C_4$ & $ \mathbb Q (\sqrt{-3}) $& $1 \times $ $ M_2\left( \mathbb Q (\sqrt{-3}) \right)$ \\
{[}12, 4{]} & $ D_{12} $ & $ \mathbb Q $& $1 \times $ $ M_2\left( \mathbb Q \right)$ \\
{[}12, 4{]} & $ D_{12} $ & $ \mathbb Q (\sqrt{-1}) $& $1 \times $ $ M_2\left( \mathbb Q (\sqrt{-1}) \right)$ \\
{[}12, 4{]} & $ D_{12} $ & $ \mathbb Q (\sqrt{-2}) $& $1 \times $ $ M_2\left( \mathbb Q (\sqrt{-2}) \right)$ \\
{[}12, 4{]} & $ D_{12} $ & $ \mathbb Q (\sqrt{-3}) $& $1 \times $ $ M_2\left( \mathbb Q (\sqrt{-3}) \right)$ \\
{[}16, 6{]} & $ C_8  \rtimes   C_2$ & $ \mathbb Q $& $1 \times $ $ M_2\left( \mathbb Q (\sqrt{-1}) \right)$ \\
{[}16, 6{]} & $ C_8  \rtimes   C_2$ & $ \mathbb Q (\sqrt{-1}) $& $2 \times $ $ M_2\left( \mathbb Q (\sqrt{-1}) \right)$ \\
{[}16, 8{]} & $ {QD}_{16} \textnormal{ (also denoted by } D_{16}^{-} \textnormal{)} $ & $ \mathbb Q $& $1 \times $ $ M_2\left( \mathbb Q (\sqrt{-2}) \right)$ \\
{[}16, 8{]} & $ {QD}_{16} \textnormal{ (also denoted by } D_{16}^{-} \textnormal{)} $ & $ \mathbb Q (\sqrt{-2}) $& $2 \times $ $ M_2\left( \mathbb Q (\sqrt{-2}) \right)$ \\
{[}16, 13{]} & $( C_4  \times   C_2)  \rtimes   C_2$ & $ \mathbb Q $& $1 \times $ $ M_2\left( \mathbb Q (\sqrt{-1}) \right)$ \\
{[}16, 13{]} & $( C_4  \times   C_2)  \rtimes   C_2$ & $ \mathbb Q (\sqrt{-1}) $& $2 \times $ $ M_2\left( \mathbb Q (\sqrt{-1}) \right)$ \\
{[}18, 3{]} & $ C_3  \times   S_3 $ & $ \mathbb Q $& $1 \times $ $ M_2\left( \mathbb Q (\sqrt{-3}) \right)$ \\
{[}18, 3{]} & $ C_3  \times   S_3 $ & $ \mathbb Q (\sqrt{-3}) $& $2 \times $ $ M_2\left( \mathbb Q (\sqrt{-3}) \right)$ \\
{[}24, 1{]} & $ C_3  \rtimes   C_8$ & $ \mathbb Q $& $1 \times $ $ M_2\left( \mathbb Q (\sqrt{-1}) \right)$ \\
{[}24, 1{]} & $ C_3  \rtimes   C_8$ & $ \mathbb Q (\sqrt{-1}) $& $2 \times $ $ M_2\left( \mathbb Q (\sqrt{-1}) \right)$ \\
{[}24, 3{]} & $ {\rm SL} (2,3)$ & $ \mathbb Q $& $1 \times $ $ M_2\left( \mathbb Q (\sqrt{-3}) \right)$ \\
{[}24, 3{]} & $ {\rm SL} (2,3)$ & $ \mathbb Q (\sqrt{-1}) $& $1 \times $ $ M_2\left( \mathbb Q (\sqrt{-1}) \right)$ \\
{[}24, 3{]} & $ {\rm SL} (2,3)$ & $ \mathbb Q (\sqrt{-2}) $& $1 \times $ $ M_2\left( \mathbb Q (\sqrt{-2}) \right)$ \\
{[}24, 3{]} & $ {\rm SL} (2,3)$ & $ \mathbb Q (\sqrt{-3}) $& $3 \times $ $ M_2\left( \mathbb Q (\sqrt{-3}) \right)$ \\
{[}24, 5{]} & $ C_4  \times   S_3 $ & $ \mathbb Q $& $1 \times $ $ M_2\left( \mathbb Q (\sqrt{-1}) \right)$ \\
{[}24, 5{]} & $ C_4  \times   S_3 $ & $ \mathbb Q (\sqrt{-1}) $& $2 \times $ $ M_2\left( \mathbb Q (\sqrt{-1}) \right)$ \\
{[}24, 8{]} & $( C_6  \times   C_2)  \rtimes   C_2$ & $ \mathbb Q $& $1 \times $ $ M_2\left( \mathbb Q (\sqrt{-3}) \right)$ \\
{[}24, 8{]} & $( C_6  \times   C_2)  \rtimes   C_2$ & $ \mathbb Q (\sqrt{-3}) $& $2 \times $ $ M_2\left( \mathbb Q (\sqrt{-3}) \right)$ \\
{[}24, 10{]} & $ C_3  \times   D_8 $ & $ \mathbb Q $& $1 \times $ $ M_2\left( \mathbb Q (\sqrt{-3}) \right)$ \\
{[}24, 10{]} & $ C_3  \times   D_8 $ & $ \mathbb Q (\sqrt{-3}) $& $2 \times $ $ M_2\left( \mathbb Q (\sqrt{-3}) \right)$ \\
{[}24, 11{]} & $ C_3  \times   Q_8 $ & $ \mathbb Q $& $1 \times $ $ M_2\left( \mathbb Q (\sqrt{-3}) \right)$ \\
{[}24, 11{]} & $ C_3  \times   Q_8 $ & $ \mathbb Q (\sqrt{-3}) $& $2 \times $ $ M_2\left( \mathbb Q (\sqrt{-3}) \right)$ \\
{[}32, 8{]} & $( C_2  \times   C_2) . ( C_4  \times   C_2)$ & $\mathbb Q$& $1 \times $ $ M_2\left(\frac{-1,-1}{\mathbb Q}\right)$ \\
{[}32, 11{]} & $( C_4  \times   C_4)  \rtimes   C_2$ & $ \mathbb Q $& $2 \times $ $ M_2\left( \mathbb Q (\sqrt{-1}) \right)$ \\
{[}32, 11{]} & $( C_4  \times   C_4)  \rtimes   C_2$ & $ \mathbb Q (\sqrt{-1}) $& $4 \times $ $ M_2\left( \mathbb Q (\sqrt{-1}) \right)$ \\
{[}32, 44{]} & $( C_2  \times   Q_8 )  \rtimes   C_2$ & $\mathbb Q$& $1 \times $ $ M_2\left(\frac{-1,-1}{\mathbb Q}\right)$ \\
{[}32, 50{]} & $( C_2  \times   Q_8 )  \rtimes   C_2$ & $\mathbb Q$& $1 \times $ $ M_2\left(\frac{-1,-1}{\mathbb Q}\right)$ \\
{[}36, 6{]} & $ C_3  \times  ( C_3  \rtimes   C_4)$ & $ \mathbb Q $& $1 \times $ $ M_2\left( \mathbb Q (\sqrt{-3}) \right)$ \\
{[}36, 6{]} & $ C_3  \times  ( C_3  \rtimes   C_4)$ & $ \mathbb Q (\sqrt{-3}) $& $2 \times $ $ M_2\left( \mathbb Q (\sqrt{-3}) \right)$ \\
{[}36, 12{]} & $ C_6  \times   S_3 $ & $ \mathbb Q $& $1 \times $ $ M_2\left( \mathbb Q (\sqrt{-3}) \right)$ \\
{[}36, 12{]} & $ C_6  \times   S_3 $ & $ \mathbb Q (\sqrt{-3}) $& $2 \times $ $ M_2\left( \mathbb Q (\sqrt{-3}) \right)$ \\
{[}40, 3{]} & $ C_5  \rtimes   C_8$ & $\mathbb Q$& $1 \times $ $ M_2\left(\frac{-2,-5}{\mathbb Q}\right)$ \\
{[}48, 16{]} & $( C_3  \rtimes   C_8)  \rtimes   C_2$ & $\mathbb Q$& $1 \times $ $ M_2\left(\frac{-1,-1}{\mathbb Q}\right)$ \\
{[}48, 18{]} & $ C_3  \rtimes   Q_{16} $ & $\mathbb Q$& $1 \times $ $ M_2\left(\frac{-1,-3}{\mathbb Q}\right)$ \\
{[}48, 28{]} & $ C_2 .  S_4  =  {\rm SL} (2,3) .  C_2 = \mathcal{O}^*$ & $\mathbb Q$& $1 \times $ $ M_2\left(\frac{-1,-3}{\mathbb Q}\right)$ \\
{[}48, 29{]} & $ {\rm GL} (2,3)$ & $ \mathbb Q $& $1 \times $ $ M_2\left( \mathbb Q (\sqrt{-2}) \right)$ \\
{[}48, 29{]} & $ {\rm GL} (2,3)$ & $ \mathbb Q (\sqrt{-2}) $& $2 \times $ $ M_2\left( \mathbb Q (\sqrt{-2}) \right)$ \\
{[}48, 33{]} & $ {\rm SL} (2,3)  \rtimes   C_2$ & $ \mathbb Q $& $1 \times $ $ M_2\left( \mathbb Q (\sqrt{-1}) \right)$ \\
{[}48, 33{]} & $ {\rm SL} (2,3)  \rtimes   C_2$ & $ \mathbb Q (\sqrt{-1}) $& $2 \times $ $ M_2\left( \mathbb Q (\sqrt{-1}) \right)$ \\
{[}48, 39{]} & $( C_2  \times  ( C_3  \rtimes   C_4))  \rtimes   C_2$ & $\mathbb Q$& $1 \times $ $ M_2\left(\frac{-1,-3}{\mathbb Q}\right)$ \\
{[}48, 40{]} & $ Q_8   \times   S_3 $ & $\mathbb Q$& $1 \times $ $ M_2\left(\frac{-1,-1}{\mathbb Q}\right)$ \\
{[}64, 37{]} & $( C_4  \times   C_2) . ( C_4  \times   C_2)$ & $\mathbb Q$& $2 \times $ $ M_2\left(\frac{-1,-1}{\mathbb Q}\right)$ \\
{[}64, 137{]} & $(( C_4  \times   C_4)  \rtimes   C_2)  \rtimes   C_2$ & $\mathbb Q$& $2 \times $ $ M_2\left(\frac{-1,-1}{\mathbb Q}\right)$ \\
{[}72, 19{]} & $( C_3  \times   C_3)  \rtimes   C_8$ & $\mathbb Q$& $2 \times $ $ M_2\left(\frac{-1,-3}{\mathbb Q}\right)$ \\
{[}72, 20{]} & $( C_3  \rtimes   C_4)  \times   S_3 $ & $\mathbb Q$& $1 \times $ $ M_2\left(\frac{-1,-3}{\mathbb Q}\right)$ \\
{[}72, 22{]} & $( C_6  \times   S_3 )  \rtimes   C_2$ & $\mathbb Q$& $1 \times $ $ M_2\left(\frac{-1,-3}{\mathbb Q}\right)$ \\
{[}72, 24{]} & $( C_3  \times   C_3)  \rtimes   Q_8 $ & $\mathbb Q$& $1 \times $ $ M_2\left(\frac{-1,-3}{\mathbb Q}\right)$ \\
{[}72, 25{]} & $ C_3  \times   {\rm SL} (2,3)$ & $ \mathbb Q $& $3 \times $ $ M_2\left( \mathbb Q (\sqrt{-3}) \right)$ \\
{[}72, 25{]} & $ C_3  \times   {\rm SL} (2,3)$ & $ \mathbb Q (\sqrt{-3}) $& $6 \times $ $ M_2\left( \mathbb Q (\sqrt{-3}) \right)$ \\
{[}72, 30{]} & $ C_3  \times  (( C_6  \times   C_2)  \rtimes   C_2)$ & $ \mathbb Q $& $2 \times $ $ M_2\left( \mathbb Q (\sqrt{-3}) \right)$ \\
{[}72, 30{]} & $ C_3  \times  (( C_6  \times   C_2)  \rtimes   C_2)$ & $ \mathbb Q (\sqrt{-3}) $& $4 \times $ $ M_2\left( \mathbb Q (\sqrt{-3}) \right)$ \\
{[}96, 67{]} & $ {\rm SL} (2,3)  \rtimes   C_4$ & $ \mathbb Q $& $2 \times $ $ M_2\left( \mathbb Q (\sqrt{-1}) \right)$ \\
{[}96, 67{]} & $ {\rm SL} (2,3)  \rtimes   C_4$ & $ \mathbb Q (\sqrt{-1}) $& $4 \times $ $ M_2\left( \mathbb Q (\sqrt{-1}) \right)$ \\
{[}96, 190{]} & $( C_2  \times   {\rm SL} (2,3))  \rtimes   C_2$ & $\mathbb Q$& $1 \times $ $ M_2\left(\frac{-1,-1}{\mathbb Q}\right)$ \\
{[}96, 191{]} & $( C_2 .  S_4  =  {\rm SL} (2,3) .  C_2 = \mathcal{O}^*)  \rtimes   C_2$ & $\mathbb Q$& $1 \times $ $ M_2\left(\frac{-1,-1}{\mathbb Q}\right)$ \\
{[}96, 202{]} & $( C_2  \times   {\rm SL} (2,3))  \rtimes   C_2$ & $\mathbb Q$& $1 \times $ $ M_2\left(\frac{-1,-1}{\mathbb Q}\right)$ \\
{[}120, 5{]} & $ {\rm SL} (2,5)$ & $\mathbb Q$& $1 \times $ $ M_2\left(\frac{-1,-3}{\mathbb Q}\right)$ \\
{[}128, 937{]} & $( Q_8   \times   Q_8 )  \rtimes   C_2$ & $\mathbb Q$& $4 \times $ $ M_2\left(\frac{-1,-1}{\mathbb Q}\right)$ \\
{[}144, 124{]} & $ C_3  \rtimes  ( C_2 .  S_4  =  {\rm SL} (2,3) .  C_2 = \mathcal{O}^*)$ & $\mathbb Q$& $3 \times $ $ M_2\left(\frac{-1,-3}{\mathbb Q}\right)$ \\
{[}144, 128{]} & $ S_3   \times   {\rm SL} (2,3)$ & $\mathbb Q$& $1 \times $ $ M_2\left(\frac{-1,-1}{\mathbb Q}\right)$ \\
{[}144, 135{]} & $(( C_3  \times   C_3)  \rtimes   C_8)  \rtimes   C_2$ & $\mathbb Q$& $4 \times $ $ M_2\left(\frac{-1,-3}{\mathbb Q}\right)$ \\
{[}144, 148{]} & $( C_2  \times  (( C_3  \times   C_3)  \rtimes   C_4))  \rtimes   C_2$ & $\mathbb Q$& $2 \times $ $ M_2\left(\frac{-1,-3}{\mathbb Q}\right)$ \\
{[}160, 199{]} & $(( C_2  \times   Q_8 )  \rtimes   C_2)  \rtimes   C_5$ & $\mathbb Q$& $1 \times $ $ M_2\left(\frac{-1,-1}{\mathbb Q}\right)$ \\
{[}192, 989{]} & $(( C_2 .  S_4  =  {\rm SL} (2,3) .  C_2 = \mathcal{O}^*)  \rtimes   C_2)  \rtimes   C_2$ & $\mathbb Q$& $2 \times $ $ M_2\left(\frac{-1,-1}{\mathbb Q}\right)$ \\
{[}240, 89{]} & $ C_2 .  S_5  =  {\rm SL} (2,5) .  C_2$ & $\mathbb Q$& $1 \times $ $ M_2\left(\frac{-2,-5}{\mathbb Q}\right)$ \\
{[}240, 90{]} & $ {\rm SL} (2,5)  \rtimes   C_2$ & $\mathbb Q$& $1 \times $ $ M_2\left(\frac{-2,-5}{\mathbb Q}\right)$ \\
{[}288, 389{]} & $(( C_3  \rtimes   C_4)  \times  ( C_3  \rtimes   C_4))  \rtimes   C_2$ & $\mathbb Q$& $2 \times $ $ M_2\left(\frac{-1,-3}{\mathbb Q}\right)$ \\
{[}320, 1581{]} & $((( C_2  \times   Q_8 )  \rtimes   C_2)  \rtimes   C_5) .  C_2$ & $\mathbb Q$& $2 \times $ $ M_2\left(\frac{-1,-1}{\mathbb Q}\right)$ \\
{[}384, 618{]} & $(( Q_8   \times   Q_8 )  \rtimes   C_3)  \rtimes   C_2$ & $\mathbb Q$& $1 \times $ $ M_2\left(\frac{-1,-1}{\mathbb Q}\right)$ \\
{[}384, 18130{]} & $(( Q_8   \times   Q_8 )  \rtimes   C_3)  \rtimes   C_2$ & $\mathbb Q$& $1 \times $ $ M_2\left(\frac{-1,-1}{\mathbb Q}\right)$ \\
{[}720, 409{]} & $ {\rm SL} (2,9)$ & $\mathbb Q$& $2 \times $ $ M_2\left(\frac{-1,-3}{\mathbb Q}\right)$ \\
{[}1152, 155468{]} & $( {\rm SL} (2,3)  \times   {\rm SL} (2,3))  \rtimes   C_2$ & $\mathbb Q$& $1 \times $ $ M_2\left(\frac{-1,-1}{\mathbb Q}\right)$ \\
{[}1920, 241003{]} & $ C_2 . (( C_2  \times   C_2  \times   C_2  \times   C_2)  \rtimes   A_5 )$ & $\mathbb Q$& $1 \times $ $ M_2\left(\frac{-1,-1}{\mathbb Q}\right)$ \\
\end{longtable}}

\section{Group algebras with exceptional components of type 1}\label{section:type1}

In this section we consider exceptional components of type 1. We provide necessary and sufficient conditions for a finite group $G$ to be $F$-critical. 

\begin{definition}
Let $G$ be a finite group and $F$ an abelian number field. We say that $G$ is $F$-\emph{critical} if and only if 
\begin{enumerate}
 \item $FG$ has a Wedderburn component which is exceptional of type 1, and
 \item for any $1\neq N \unlhd G$ the group algebra $F(G/N)$ does not have a Wedderburn component which is exceptional of type 1.
\end{enumerate}
\end{definition}
%

Note that if a group $G$ is $F$-critical with corresponding exceptional component $B$, the $F$-representation of $G$ associated to $B$ is necessarily faithful. In particular $G$ is in the classification of Amitsur (cf.\ \cref{SubgruposAD}).

We consider the NZ-groups from Amitsur's classification in \cref{O,quaternion,SL23}.

\begin{proposition}\label{O}
 Let $F$ be any abelian number field. Then neither $\OO^*$, the binary octahedral group, nor $\SL(2,5)$ is $F$-critical.
\end{proposition}

\begin{proof}
 We give a proof for $\SL(2,5)$, the arguments for $\OO^*$ are similar.\\
 The Wedderburn decomposition of $\Q\SL(2,5)$ equals $$ \Q \oplus M_4(\Q) \oplus\quat{-1,-1}{\Q(\sqrt{5})}\oplus M_2\quat{-1,-3}{\Q} \oplus M_5(\Q) \oplus M_3\quat{-1,-1}{\Q}\oplus M_3(\Q(\sqrt{5})).$$ 
The only possible exceptional component of type 1 of $F\SL(2,5)$ comes from \[ F\otimes_{\Q} \quat{-1,-1}{\Q(\sqrt{5})}=\quat{-1,-1}{F(\sqrt{5})}^{[F\cap \Q(\sqrt{5}):\Q]}.\] 
The algebra $\quat{-1,-1}{\Q(\sqrt{5})}$ has all local Schur indices 1, except $m_{\infty}\quat{-1,-1}{\Q(\sqrt{5})}=2$.
By \cref{splitting}, $\quat{-1,-1}{F(\sqrt{5})}$ is therefore only a division algebra when $F$ is totally real. But in this case it is a totally definite quaternion algebra and hence not exceptional.
\end{proof}

\begin{proposition}\label{SL23}
Let $F$ be an abelian number field.
\begin{enumerate}
\item $\SL(2,3)$ is $F$-critical if and only if $F$ is totally imaginary and both $e_2(F/\Q)$ and $f_2(F/\Q)$ are odd. In this case, $F\SL(2,3)$ has an exceptional component $\quat{-1,-1}{F}$.

\item Let $M$ be a group in (Z) of order coprime to $6$, such that $2$ has odd order modulo $|M|$. Then 
$\SL(2,3)\times M$ is $F$-critical if and only if $M$ is a cyclic group of prime order $p$, $F$ is totally real and both $e_2(F(\zeta_p)/\Q)$ and $f_2(F(\zeta_p)/\Q)$ are odd.
In this case, $F(\SL(2,3)\times C_p)$ has an exceptional component $\quat{-1,-1}{F(\zeta_p)}$.
\end{enumerate} 
\end{proposition}

\begin{proof}
(1) Let $G=\SL(2,3)$. We first assume that $G$ is $F$-critical. The Wedderburn decomposition of $\Q G$ equals $$\Q G\iso \Q\oplus \Q (\zeta_3)\oplus M_3(\Q)\oplus \quat{-1,-1}{\Q}\oplus M_2(\Q(\zeta_3)).$$ The only exceptional component of $FG$ can be $F\otimes_{\Q}\quat{-1,-1}{\Q}=\quat{-1,-1}{F}$, hence it is an exceptional component of $FG$. Using \cref{Shirvani-2.1.9} the result follows.

Now we prove the converse. By \cref{Shirvani-2.1.9} we have that $\quat{-1,-1}{F}$ is an exceptional component of $FG$ of type 1. On the other hand, $G$ has only one non-abelian proper quotient which is isomorphic to $A_4$, and the Wedderburn decomposition of $\Q A_4$ equals $$\Q A_4\iso \Q\oplus \Q (\zeta_3)\oplus M_3(\Q).$$ Hence the group algebra of any non-abelian proper quotient of $G$ does not have exceptional components. So we conclude that $G$ is $F$-critical.

(2) Let $G=\SL(2,3)\times M$. We first assume that $G$ is $F$-critical. Observe that in the Wedderburn decomposition of $FM$ totally definite quaternion algebras can not appear, since the order of $M$ is odd. Moreover, due to the fact that $M$ is a proper quotient of $G$, $FM$ does not have non-commutative division algebras as simple components.  Another non-abelian quotient of $G$ is $\SL(2,3)$, and according to the Wedderburn decomposition of $\Q \SL(2,3)$, $\quat{-1,-1}{F}$ has to be either a totally definite quaternion algebra or a $2\times 2$-matrix ring over $F$. Suppose that $\quat{-1,-1}{F}$ is a $2\times 2$-matrix ring over $F$. The fact that $FG\iso F\SL(2,3)\otimes_F FM$ implies that there is not any division algebra in the Wedderburn decomposition of $FG$ which is a contradiction. Therefore, $\quat{-1,-1}{F}$ is a totally definite quaternion algebra and $F$ is a totally real field.

On the other hand, let $D$ be an exceptional component of type 1 of $FG$, then $D\iso D_1\otimes_F D_2$ where $D_1$ and $D_2$ are simple components of $F\SL(2,3)$ and $FM$ respectively. Having in mind the Wedderburn decompositions of $F\SL(2,3)$ and $FM$, and since $D$ is a division algebra which is not a totally definite quaternion algebra, we can deduce that $D_1\iso \quat{-1,-1}{F}$ and $D_2\iso F(\zeta_d)$ for some divisor $d$ of the order of $M$, $d\neq 1$. We know that $F(\zeta_d)$ is a simple component of $F(M/M')$ (and so $D$ is a simple component of $F(G/M')$), hence by hypothesis $M'$ is trivial, so that $M$ is abelian and by the conditions in \cref{SubgruposAD}, $M$ is a cyclic group. Now we claim that $M$ has prime order. Let $d$ be a proper divisor of $|M|$, then $\quat{-1,-1}{F}\otimes_F F(\zeta_d)$ is a simple component of $F(G/C_{|M|/d})$. By hypothesis it must be a $2\times 2$-matrix ring over $F(\zeta_d)$, and it follows that $D\iso \quat{-1,-1}{F}\otimes_F F(\zeta_{|M|})$. By \cref{Shirvani-2.1.9}, $F(\zeta_d)$ is a totally imaginary field and $e_2(F(\zeta_d)/\Q)$ or $f_2(F(\zeta_d)/\Q)$ is even, moreover both $e_2(F(\zeta_{|M|})/\Q)$ and $f_2(F(\zeta_{|M|})/\Q)$ are odd. By \cref{reciprocity}, we have $$e_2(F(\zeta_{|M|})/\Q)=e_2(F(\zeta_{|M|})/F(\zeta_d))e_2(F(\zeta_d)/\Q)$$ and $$f_2(F(\zeta_{|M|})/\Q)=f_2(F(\zeta_{|M|})/F(\zeta_d))f_2(F(\zeta_d)/\Q)$$ are both odd, but this is a contradiction since $e_2(F(\zeta_d)/\Q)$ or $f_2(F(\zeta_d)/\Q)$ is even. So the claim follows.

By the above paragraph, we have that $G=\SL(2,3)\times C_p$ and $D\iso \quat{-1,-1}{F}\otimes_F F(\zeta_p)$ is an exceptional component of type 1 of $FG$. Then again by \cref{Shirvani-2.1.9} both $e_2(F(\zeta_p)/\Q)$ and $f_2(F(\zeta_p)/\Q)$ are odd. 

Now suppose that $G=\SL(2,3)\times C_p$.  Using \cref{Shirvani-2.1.9} and having in mind the Wedderburn decomposition of $\Q \SL(2,3)$, we can deduce that $D\iso \quat{-1,-1}{F}\otimes_F F(\zeta_p)$ is the unique exceptional component of $FG$ of type 1. Note that the non-abelian proper quotients of $G$ are $\SL(2,3)$, $A_4$ and $A_4\times C_p$. Due to the fact that $F$ is a totally real field, $F\SL(2,3)$ does not have exceptional components of type 1. As in the Wedderburn decomposition of $\Q A_4$ and $\Q (A_4\times C_p)$ only fields and matrix rings show up, we have that $F A_4$ and $F(A_4\times C_p)$ do not have division algebras as simple components. This finishes the proof.
\end{proof}


\begin{proposition}\label{quaternion}
 Let $F$ be an abelian number field and $Q_{4k}$ be the generalized quaternion group. 
 \begin{enumerate}
  \item Let $k$ be even. Then $Q_{4k}$ is $F$-critical if and only if $k=2$, $F$ is totally imaginary and both $e_2(F/\Q)$ and $f_2(F/\Q)$ are odd. In this case $F Q_{8}$ has an exceptional component $\quat{-1,-1}{F}$.
 \item Let $M$ be a group in (Z) of odd order such that $2$ has odd order modulo $|M|$. Then $Q_8\times M$ is $F$-critical if and only if $M$ is a cyclic group of prime order $p$, $F$ is totally real and both $e_2(F(\zeta_p)/\Q)$ and $f_2(F(\zeta_p)/\Q)$ are odd.
In this case $F(Q_8\times C_p)$ has an exceptional component $\quat{-1,-1}{F(\zeta_p)}$.
\end{enumerate}
\end{proposition}
\begin{proof}
 (1) Let $Q_{4k}=\GEN{a,b\mid a^{2k}=1,a^k=b^2,b\inv ab=a\inv}$ with $k=2^t k'$, $t\ge 1$ and $2\nmid k'$, then by \cref{SSPmetacyclic}, the non-commutative components of $\Q Q_{4k}$ come from the strong Shoda pairs $(\GEN{a},\GEN{a^d})$ with $d\mid 2k$ such that $d\neq 1,2$. The corresponding simple components are 
\begin{eqnarray*} \Q Q_{4k}e(Q_{4k},\GEN{a},\GEN{a^d}) = \left\{\begin{array}{ll} (\Q(\zeta_{d})/\Q(\zeta_{d}+\zeta_{d}^{-1}),1)= M_2(\Q(\zeta_d+\zeta_d^{-1})), & \mbox {if } d\mid k \\
 (\Q(\zeta_{d})/\Q(\zeta_{d}+\zeta_{d}^{-1}),-1)= \quat{-1,-1}{\Q(\zeta_{d}+\zeta_{d}^{-1})}, & \mbox {otherwise.} \end{array}\right.\end{eqnarray*}
 
Now suppose that $Q_{4k}$ is $F$-critical and $k'\neq 1$.  By the above paragraph the exceptional component of $F Q_{4k}$ comes from 
\begin{eqnarray*}
F\otimes_{\Q} \quat{-1,-1}{\Q(\zeta_{d}+\zeta_{d}^{-1})}
&=&\quat{-1,-1}{F(\zeta_{d}+\zeta_{d}^{-1})}^{[F\cap \Q(\zeta_{d}+\zeta_{d}^{-1}):\Q]}\\ 
&=&\left(F(\zeta_{d}+\zeta_{d}^{-1})\otimes_{\Q}\quat{-1,-1}{\Q}\right)^{[F\cap \Q(\zeta_{d}+\zeta_{d}^{-1}):\Q]}
\end{eqnarray*} 
for some $d$ satisfying that $d\mid 2k$ and $d\nmid k$. Note that this implies that $d=2^{t+1} k_1'$ where $k_1'\mid k'$. This quaternion algebra can not be totally definite (i.e. $F$ is not totally real) although it is a division algebra. Hence, by \cref{Shirvani-2.1.9}, $e_2(F(\zeta_{d}+\zeta_{d}^{-1})/\Q)$ and $f_2(F(\zeta_{d}+\zeta_{d}^{-1})/\Q)$ are odd. Due to $k'\neq 1$, we have that $Q_{4 \cdot 2^t}$ is a non-abelian proper quotient of $Q_{4k}$, moreover $ \quat{-1,-1}{F(\zeta_{2^{t+1}}+\zeta_{2^{t+1}}^{-1})}$ is a simple component of $F Q_{4 \cdot 2^t}$, let this algebra be denoted by $B$. Hence $B$ must be either a totally definite quaternion algebra or a $2\times 2$-matrix ring over $F(\zeta_{2^{t+1}}+\zeta_{2^{t+1}}^{-1})$. If $F$ splits the quaternion algebra $\quat{-1,-1}{\Q(\zeta_{2^{t+1}}+\zeta_{2^{t+1}}^{-1})}$, by \cref{Shirvani-2.1.9}, $e_2(F(\zeta_{2^{t+1}}+\zeta_{2^{t+1}}^{-1})/\Q)$ or ${f_2(F(\zeta_{2^{t+1}}+\zeta_{2^{t+1}}^{-1})/\Q)}$ is even. Applying \cref{reciprocity} for $e_2(F(\zeta_{d}+\zeta_{d}^{-1})/\Q)$ and $f_2({F(\zeta_{d}+\zeta_{d}^{-1})}/\Q)$ gives a contradiction.
Therefore $B$ is a totally definite quaternion algebra, and then $F(\zeta_{2^{t+1}}+\zeta_{2^{t+1}}^{-1})$ is totally real, and hence $F$ is totally real, again a contradiction. Therefore $k'=1$. On the other hand, by the description of the simple components above, $B$ is the only candidate for being an exceptional component of type 1. 

 By \cref{Shirvani-2.1.9}, $B$ is a division algebra whenever $F$ is totally real or $e_2(F(\zeta_{2^{t+1}}+\zeta_{2^{t+1}}^{-1})/\Q)$ and $f_2(F(\zeta_{2^{t+1}}+\zeta_{2^{t+1}}^{-1})/\Q)$ are odd. Because, for $t>1$, ${F(\zeta_{8}+\zeta_{8}^{-1})\subseteq F(\zeta_{2^{t+1}}+\zeta_{2^{t+1}}^{-1})}$ and $e_2(F(\zeta_{8}+\zeta_{8}^{-1})/\Q)=e_2(\Q(\sqrt{2})/\Q)=2$, the ramification index $e_2(F(\zeta_{2^{t+1}}+\zeta_{2^{t+1}}^{-1})/\Q)$ is even. But in this case $B$ is a totally definite quaternion algebra and hence not exceptional. We conclude that $t$ must be equal to 1 and $k=2$. Since $G$ is $F$-critical, $\quat{-1,-1}{F}$ is not a totally definite quaternion algebra. Hence $F$ is totally imaginary and $e_2(F/\Q)$ and $f_2(F/\Q)$ are odd.

Finally we prove the converse. By the assumptions and \cref{Shirvani-2.1.9} we have that $ \quat{-1,-1}{F}$ is an exceptional component of type 1 of $F Q_{8}$. 
The only proper quotients of $Q_{8}$ are abelian groups and hence $Q_{8}$ is $F$-critical.

(2) follows as in the proof of \cref{SL23} (2).
\end{proof}

Now we consider the Z-groups from Amitsur's classification.

Because of the following theorem, suggested by \cite{Amitsur}, one can discover a minimal faithful division algebra component in the Wedderburn decomposition of the groups in (Z).
\begin{proposition}\label{minimal}
 Let $F$ be a number field and $G=\GEN{a}_m \rtimes_k \GEN{b}_n$ with $\gcd(n,m)=1$. Let ${C\in\mathcal{C}_F(\GEN{ab^{n/k}})}$. If $FGe_C(G,\GEN{ab^{n/k}},1)$ is not a division algebra, then $FG$ does not have a division algebra as a faithful simple component. Furthermore, if $FG$ has a division algebra $D$ as a faithful simple component, then $FGe_C(G,\GEN{ab^{n/k}},1)$ is embedded in $D$ and in particular ${E_F(G,\GEN{ab^{n/k}})=G}$, $F(\zeta_{mk})*G/\GEN{ab^{n/k}}$ is a division algebra and $D$ has degree at least $\frac{n}{k}$.
\end{proposition}
\begin{proof}
Let $A=\GEN{ab^{n/k}}$, $E=E_F(G,A)$ and $C\in\mathcal{C}_F(A)$. Then $(A,1)$ is a strong Shoda pair of $G$ and, by \cref{main}, $$FGe_C(G,A,1)=M_{[G:E]}(F(\zeta_{mk})*E/A)$$ is a Wedderburn component of $F G$ which is a direct factor of $F \otimes_{\Q} \Q Ge(G,A,1)$. It is easy to check that $FGe_C(G,A,1)$ is a faithful component of $FG$.

Assume now that $FG$ has as a faithful simple component a division algebra $D$ different from $FGe_C(G,A,1)$. Then also $\Q G$ has a division algebra $D'$ different from $\Q Ge(G,A,1)$ as a faithful simple component. By the minimality of $\Q Ge(G,A,1)$ as explained in \cite[Lemma 2.3]{2013Caicedo}, $\Q Ge(G,A,1)$ is embedded in $D'$. But this means that $FGe_C(G,A,1)$ is a direct factor of ${F\otimes_{\Q}\Q Ge(G,A,1)}$, which in its turn is embedded in $F\otimes_{\Q} D'$.
If $FGe_C(G,A,1)$ is not a division algebra, then it is a matrix ring and it has nilpotent elements. But then also $F\otimes_{\Q} D'$ has nilpotent elements. Since $F\otimes_{\Q} D'$ is a direct sum of isomorphic copies of $D$, also $D$ has nilpotent elements, which is a contradiction. 

We conclude that $FGe_C(G,\GEN{ab^{n/k}},1)$ is a division algebra of degree $\frac{n}{k}$, which is embedded in $D$. Necessarily $E=G$ and $D$ has degree at least $\frac{n}{k}$.
\end{proof}


We investigate when $E_F(G,H/K)=G$ for the strong Shoda pair $(\GEN{ab^{n/k}},1)$.
\begin{lemma}\label{E}
 Let $F$ be a number field and $G=\GEN{a}_m \rtimes_k \GEN{b}_n$ with $\gcd(m,n)=1$, $b^{-1}ab=a^r$. Then $E_F(G,\GEN{ab^{n/k}})=G$ if and only if $\Q(\zeta_m)\cap F$ is contained in $\Fix(\Q(\zeta_m)\rightarrow \Q(\zeta_m):\zeta_m\mapsto \zeta_m^r)$. 
\end{lemma}
\begin{proof}
 Let $r'$ be such that $r'\equiv r \mod m$ and $r'\equiv 1\mod k$. Then, $E_F(G,\GEN{ab^{n/k}}) =G$ if and only if $\GEN{r'}\subseteq I_{mk}(F)$. This happens if and only if $\sigma:\zeta_{mk}\mapsto \zeta_{mk}^{r'}$ is a Galois automorphism of the extension $F(\zeta_{mk})/F$, which is equivalent with $F$ being in the fixed field of $\sigma$. Since $\sigma$ fixes $\zeta_k$, this again is equivalent with $\Q(\zeta_m)\cap F\subseteq \Fix(\Q(\zeta_m)\rightarrow \Q(\zeta_m):\zeta_m\mapsto \zeta_m^r)$.
\end{proof}

We investigate the structure of the simple algebra $FGe_C(G,\GEN{ab^{n/k}},1)$.
\begin{lemma}\label{structure}
 Let $F$ be a number field and $G=\GEN{a}_m \rtimes_k \GEN{b}_n$ with $\gcd(m,n)=1$, ${b^{-1}ab=a^r}$ and $\Q(\zeta_m)\cap F\subseteq \Fix(\Q(\zeta_m)\rightarrow \Q(\zeta_m):\zeta_m\mapsto \zeta_m^r)$. Let $\sigma:F(\zeta_{mk})\rightarrow F(\zeta_{mk}):\zeta_m\mapsto \zeta_m^r$, $A=\GEN{ab^{n/k}}$, $C\in\mathcal{C}_F(A)$ and $K=\Fix(\sigma)$. Then $$FGe_C(G,\GEN{ab^{n/k}},1)=(K(\zeta_{m})/K,\sigma,\zeta_k).$$

If furthermore $n=4$ and $k=2$, then $r\equiv -1 \mod m$ and $$FGe_C(G,\GEN{ab^{n/k}},1)=\quat{-1,(\zeta_m-\zeta_m^{-1})^2}{F(\zeta_m+\zeta_m^{-1})}.$$
\end{lemma}
\begin{proof}
 From \cref{E,main} it follows that \[FGe_C(G,\GEN{ab^{n/k}},1)=F(\zeta_{mk})*G/A=\sum_{i=0}^{\frac{n}{k}-1}F(\zeta_{mk})u_{\sigma^i}\] with $u_{\sigma^i}\zeta_{m}=\zeta_{m}^{r^i}u_{\sigma^i}$, $u_{\sigma^i}\zeta_{k}=\zeta_{k}u_{\sigma^i}$ and $u_{\sigma^i}^{n/k}=\zeta_k$. The center of this simple algebra is clearly equal to $K=\Fix(\sigma)$ and therefore we can denote it as the cyclic cyclotomic algebra $FGe_C(G,\GEN{ab^{n/k}},1)=(K(\zeta_{m})/K,\sigma,\zeta_k)$.

If $n=4$ and $k=2$, then the degree of $FGe_C(G,\GEN{ab^{n/k}},1)$ is 2 and hence it is a quaternion algebra over its center $F(\zeta_m+\zeta_m^{-1})$. An easy computation shows that $FGe_C(G,\GEN{ab^{n/k}},1)$ is generated over its center by elements $x$ and $y$ satisfying $x^2=\zeta_2=-1$ and $y^2=(\zeta_m-\zeta_m^{-1})^2$ and $xy=-yx$.
\end{proof}

\begin{lemma}\label{reduction}
 Let $N\unlhd G$, and denote by $\overline{\phantom{m}}$ the reduction mod $N$. If $(H,1)$ is a strong Shoda pair of $G$ and $(\overline{H},1)$ is a strong Shoda pair of $\overline{G}$, then $\overline{E_F(G,H)}\subseteq E_F(\overline{G},\overline{H})$.
\end{lemma}
\begin{proof}
This follows directly from the fact that if $\zeta_{|H|}\mapsto \zeta_{|H|}^i$ determines an automorphism of $\Gal(F(\zeta_{|H|})/F)$, then it restricts to an automorphism $\zeta_{[HN:N]}\mapsto \zeta_{[HN:N]}^i$ in $\Gal(F(\zeta_{[HN:N]})/F)$. 
\end{proof}

Because of the structure of metacyclic groups $\GEN{a}_m \rtimes_k \GEN{b}_n$ and their group algebras, we can deduce that, for $m$ a prime, the minimal faithful division algebra component in \cref{minimal} is essentially the only possible faithful division algebra showing up in the Wedderburn decomposition.
\begin{proposition}\label{unique}
 Let $F$ be a number field and $G=\GEN{a}_p \rtimes_k \GEN{b}_n$ with $p$ a prime not dividing $n$. The only possible faithful division algebra components of $FG$ are $FGe_C(G,\GEN{ab^{n/k}},1)$, with $C\in\mathcal{C}_F(\GEN{ab^{n/k}})$. 
\end{proposition}
\begin{proof}
Let $r$ be such that $b^{-1}ab=a^r$. By \cref{SSPmetacyclic}, we deduce that the only non-commutative components of $FG$ are determined by the strong Shoda pairs $(G_d,K)$ with $1\neq d\mid \frac{n}{k}$, $G_d=\GEN{a,b^d}$, $G_d/K$ cyclic and $d=\min\{x\mid \frac{n}{k} : a^{r^x-1}\in K\}$. Let $A=G_{n/k}=\GEN{ab^{n/k}}$.

Assume that such a strong Shoda pair leads to a faithful division algebra component, then $E_F(G,G_d/K)=G$ and $FG e_C(G,G_d,K)$ has degree at least $\frac{n}{k}$ by \cref{minimal}. However its degree equals $|G/G_d|=d$. Therefore $d=\frac{n}{k}$, $G_d=A$ and $K=\GEN{b^{ln/k}}\subseteq Z(G)$ for $l\mid k$.

Since $(A, K)$ is a Shoda pair and the characteristic subgroup $K$ is contained in the kernel of the character associated to $e_C(G,A,K)$, in order for $FG e_C(G,A,K)$ to be a faithful component, $K$ has to be $1$.
\end{proof}

Clearly the groups of type (Z)(a) are never $F$-critical because the groups are abelian.

We study the groups of type (Z)(b).
\begin{theorem}\label{Zb}
 Let $F$ be an abelian number field and let $G$ be a finite group. Then $G$ is a Z-group of type (Z)(b) and $F$-critical if and only if $G=C_p\rtimes_2 C_4$ with $p$ a prime number satisfying ${p\equiv -1 \mod 4}$, $F$ is totally imaginary, $\Q(\zeta_{p})\cap F \subseteq \Q(\zeta_p+\zeta_p^{-1})$ and both $e_{p}(F/\Q)$ and $f_{p}(F/\Q)$ are odd. 

In this case $FG$ has an exceptional component $\quat{-1,(\zeta_{p}-\zeta_{p}^{-1})^2}{F(\zeta_{p}+\zeta_{p}^{-1})}$.
\end{theorem}
\begin{proof}
Let $G=\GEN{a}_m \rtimes_2 \GEN{b}_4$ with $b^{-1}ab = a^{-1}$ and $m$ odd. Let $F$ be an abelian number field. Assume that $G$ is $F$-critical and let $A=\GEN{ab^2}$. Then $E_F(G,A)=G$ and $F Ge_C(G,A,1)= \quat{-1,(\zeta_{m}-\zeta_{m}^{-1})^2}{F(\zeta_{m}+\zeta_{m}^{-1})}$ is a division algebra for any $C\in\mathcal{C}_F(A)$ because of \cref{structure,minimal}. By \cref{E}, $E_F(G,A)=G$ is equivalent with $\Q(\zeta_{m})\cap F \subseteq \Q(\zeta_{m}+\zeta_{m}^{-1})$. Assume that $m$ is not prime and choose a prime divisor $p$ of $m$. Then $N=\GEN{a^p}\unlhd G$ and we will use bar notation for reduction modulo $N$. Now $\overline{G} = \GEN{\overline{a}}_p\rtimes_2\GEN{\overline{b}}_4$. By \cref{reduction}, $E_F(\overline{G},\overline{A})=\overline{G}$. Therefore, for any $D\in \mathcal{C}_F(\overline{A})$, $F \overline{G} e_D(\overline{G},\overline{A},1) = \quat{-1,(\zeta_p-\zeta_p^{-1})^2}{F(\zeta_p+\zeta_p^{-1})}$ 
is itself a division algebra. Indeed, suppose that $\quat{-1,(\zeta_p-\zeta_p^{-1})^2}{F(\zeta_p+\zeta_p^{-1})}$ is not a division algebra, then $-x^2-(\zeta_p-\zeta_p^{-1})^2y^2=z^2$ has a non-zero solution in $F(\zeta_p+\zeta_p^{-1})^3\subseteq F(\zeta_m+\zeta_m^{-1})^3$, but then also $-x^2-(\zeta_m-\zeta_m^{-1})^2y^2=z^2$ has a non-zero solution in $F(\zeta_m+\zeta_m^{-1})^3$, a contradiction. Since $G$ is $F$-critical, $F \overline{G} e_D(\overline{G},\overline{A},1)$ has to be a totally definite quaternion algebra and hence $F(\zeta_{p}+\zeta_{p}^{-1})$ is totally real. But this means that $F$ is totally real. 

Since $G$ is by assumption $F$-critical, one of the strong Shoda pairs $(A,K)$ with $K\subset A$ and $a\notin K$ should produce a division algebra which is not a totally definite quaternion algebra (see \cref{SSPmetacyclic}). The center of such a $FGe_C(G,A,K)$ equals $F(\zeta_{[A:K]}+\zeta_{[A:K]}^{-1})$ and is not totally real (since if $F(\zeta_{[A:K]}+\zeta_{[A:K]}^{-1})$ is totally real then $-1$ and $(\zeta_{[A:K]}-\zeta_{[A:K]}^{-1})^2$ are totally negative). Therefore $F$ is not totally real, a contradiction. 

Hence $F$ is totally imaginary, $m=p$ prime and 
\begin{eqnarray*}
 F Ge_C(G,A,1) &=& (F(\zeta_{p})/F(\zeta_p+\zeta_p^{-1}),\zeta_p\mapsto \zeta_p^{-1},-1)\\
 &=& \quat{-1,(\zeta_{p}-\zeta_{p}^{-1})^2}{F(\zeta_{p}+\zeta_{p}^{-1})}\\
&=&F(\zeta_{p}+\zeta_{p}^{-1})\otimes_{\Q(\zeta_{p}+\zeta_{p}^{-1})}\quat{-1,(\zeta_{p}-\zeta_{p}^{-1})^2}{\Q(\zeta_{p}+\zeta_{p}^{-1})}
\end{eqnarray*}
is a division algebra by \cref{unique}. By \cref{splitting,olteanu}, $m_{p}\quat{-1,(\zeta_{p}-\zeta_{p}^{-1})^2}{\Q(\zeta_{p}+\zeta_{p}^{-1})} \not= 1$ and both $e_{p}(F(\zeta_{p}+\zeta_{p}^{-1})/\Q(\zeta_{p}+\zeta_{p}^{-1}))$ and $f_{p}(F(\zeta_{p}+\zeta_{p}^{-1})/\Q(\zeta_{p}+\zeta_{p}^{-1}))$ are odd. 

By \cref{hermanodd}: $$m_{p}\quat{-1,(\zeta_{p}-\zeta_{p}^{-1})^2}{\Q(\zeta_{p}+\zeta_{p}^{-1})} = \min\left\{l\in \N \left| \frac{p^f-1}{\gcd(p^f-1,e)} \equiv 0 \mod \frac{2}{\gcd(2,l)}\right.\right\},$$ where $e=e_p(\Q(\zeta_p)/\Q(\zeta_p+\zeta_p^{-1}))=2$ and $f=f_p(\Q(\zeta_p+\zeta_p^{-1})/\Q)=1$. Then ${m_{p}\quat{-1,(\zeta_{p}-\zeta_{p}^{-1})^2}{\Q(\zeta_{p}+\zeta_{p}^{-1})} \not= 1}$ if and only if $p\equiv -1\mod 4$. Also $e_p(\Q(\zeta_{p}+\zeta_{p}^{-1})/\Q)=\frac{p-1}{2}$ is odd since $p\equiv -1\mod 4$, $f_p(\Q(\zeta_{p}+\zeta_{p}^{-1})/\Q)=1$ and both $e_{p}(F(\zeta_{p}+\zeta_{p}^{-1})/F)$ and $f_{p}(F(\zeta_{p}+\zeta_{p}^{-1})/F)$ divide $\frac{p-1}{2}$ which is odd. Therefore $e_{p}(F(\zeta_{p}+\zeta_{p}^{-1})/\Q(\zeta_{p}+\zeta_{p}^{-1}))$ and $f_{p}(F(\zeta_{p}+\zeta_{p}^{-1})/\Q(\zeta_{p}+\zeta_{p}^{-1}))$ are odd if and only if $e_p(F/\Q)$ and $f_p(F/\Q)$ are odd. We conclude that $G$ is as in the statement of the theorem. 

Assume now that $G$ and $F$ are as in the statement, then clearly $G$ is a Z-group of type (Z)(b). Let $A=\GEN{ab^2}$. Since $\Q(\zeta_{p})\cap F \subseteq \Q(\zeta_p+\zeta_p^{-1})$, $E_F(G,A)=G$ by \cref{E}. Therefore $FG e_C(G,A,1)=\quat{-1,(\zeta_p-\zeta_p^{-1})^2}{F(\zeta_p+\zeta_p^{-1})}$ for any $C\in\mathcal{C}_F(A)$. This simple component has degree 2 over its center $F(\zeta_{p}+\zeta_{p}^{-1})$, which is totally imaginary. Therefore it is not a totally definite quaternion algebra. Due to, $p\equiv -1 \mod 4$,  $m_{p}\quat{-1,(\zeta_{p}-\zeta_{p}^{-1})^2}{\Q(\zeta_{p}+\zeta_{p}^{-1})} \not= 1$ and because of the assumptions both $e_{p}(F(\zeta_{p}+\zeta_{p}^{-1})/\Q(\zeta_{p}+\zeta_{p}^{-1}))$ and $f_{p}(F(\zeta_{p}+\zeta_{p}^{-1})/\Q(\zeta_{p}+\zeta_{p}^{-1}))$ are odd. By \cref{splitting}, the simple component $\quat{-1,(\zeta_p-\zeta_p^{-1})^2}{F(\zeta_p+\zeta_p^{-1})}$ is now a division algebra. Furthermore $G$ has only quotients isomorphic to cyclic and dihedral groups, which have only fields and matrix rings as simple components over $F$. Therefore $G$ is $F$-critical. 
\end{proof}

To conclude we look at the groups of type (Z)(c).
\begin{theorem}\label{Zc}
 Let $F$ be an abelian number field and let $G$ be a finite group. Then $G$ is a Z-group of type (Z)(c) and $F$-critical if and only if one of the following holds:
\begin{enumerate}[label=(\alph*),ref=\alph*]
 \item\label{Zc1} $G=C_q \times (C_p\rtimes_2 C_4)$ with $q$ and $p$ different odd prime numbers with $o_q(p)$ odd and ${p\equiv -1 \mod 4}$. Moreover $F$ is totally real and both $e_{p}(F(\zeta_q)/\Q)$ and $f_{p}(F(\zeta_q)/\Q)$ are odd. In this case $FG$ has an exceptional component $\quat{-1,(\zeta_{p}-\zeta_{p}^{-1})^2}{F(\zeta_q,\zeta_{p}+\zeta_{p}^{-1})}$;

 \item\label{Zc2} $G=\GEN{a}_p\rtimes_k \GEN{b}_n$ with $n\geq 8$, $p$ an odd prime numbers not dividing $n$, $b^{-1}ab=a^r$, and both $k$ and $\frac{n}{k}$ are divisible by all the primes dividing $n$. Moreover $\Q(\zeta_p)\cap F\subseteq \Q(\zeta_p+\zeta_p^r+...+\zeta_p^{r^{\frac{n}{k}-1}})$, $m_p=\frac{n}{k}$ and one of the following holds:

\begin{enumerate}[label=(\roman*),ref=\roman*]
\item \label{zc1} either $p\equiv 1 \mod 4$ or $n$ is odd, $v_q(p-1)\le v_q(k)$ for every prime divisor $q$ of $n$ and $m_{p,h}< \frac{n}{k}$ for every $h\ne k$ divisor of $k$ such that $v_q(p-1)\le v_q(h)$ for every prime divisor $q$ of $n$. If $n=2k$ and $F$ is totally imaginary, then $m_{p,2}=1$.
\item \label{zc2} $p\equiv -1 \mod 4$, $v_2(k)=1$, $v_2(n)=2$, $v_q(p-1)\le v_q(k)$ for every odd prime divisor $q$ of $n$ and $m_{p,h}< \frac{n}{k}$ for every $h\ne k$ divisor of $k$ such that $v_2(h)=1$ and $v_q(p-1)\le v_q(h)$ for every odd prime divisor $q$ of $n$.
\item \label{zc3}$p\equiv -1 \mod 4$, $v_2(p+1)+1\le v_2(k)$, $v_2(n)=v_2(k)+1$, $v_q(p-1)\le v_q(k)$ for every odd prime divisor $q$ of $n$ and 
\begin{enumerate}[label=(\arabic*),ref=\arabic*]
\item $m_{p,h}< \frac{n}{k}$ for every $h\ne k$ divisor of $k$ such that $v_2(p+1)+1\le v_2(h)$ and $v_q(p-1)\le v_q(k)$ for every odd prime divisor $q$ of $n$,
\item $m_{p,h}< \frac{n}{k}$ for every divisor $h$ of $k$  different from $2$ and $k$ such that $v_2(h)=1$ and $v_q(p-1)\le v_q(h)$ for every odd prime divisor $q$ of $n$. If $n=2k$ and $F$ is totally imaginary, then $m_{p,2}=1$.
\end{enumerate}
\end{enumerate}

Here $$m_p=\min\left\{l\in\N \left| \frac{p^f-1}{\gcd(p^f-1,e)} \equiv 0 \mod \frac{k}{\gcd(k,l)}\right.\right\}$$
with 
$K=F(\zeta_k,\zeta_p+\zeta_p^r+...+\zeta_p^{r^{\frac{n}{k}-1}})$, $e=e_p(F(\zeta_{pk})/K)$ and $f=f_p(K/\Q)$, and $$m_{p,h}=\min\left\{l\in\N \left| \frac{p^{f_h}-1}{\gcd(p^{f_h}-1,e_h)} \equiv 0 \mod \frac{h}{\gcd(h,l)}\right.\right\}$$ with $K_h=K\cap F(\zeta_{ph})$, 
$e_h=e_p(F(\zeta_{ph})/K_h)$ and $f_h=f_p(K_h/\Q)$. 

In this case $FG$ has an exceptional component $(K(\zeta_{p})/K,\sigma,\zeta_{k})$, where the action is defined by $\sigma: F(\zeta_{pk})\rightarrow F(\zeta_{pk}):\zeta_p\mapsto \zeta_p^r;\zeta_k\mapsto \zeta_k$.
\end{enumerate}
\end{theorem}
\begin{proof}
If a direct product $G\times H$ is $F$-critical then the simple division algebra component equals $F(G\times H)e=FGe_1\otimes_{F} FHe_2$ for some idempotents $e,e_1,e_2$ and both $FGe_1$ and $FHe_2$ are division algebras (including fields). Assume that $G$ and $H$ have coprime orders. If both $FGe_1$ and $FHe_2$ are non-commutative division algebras, then at least one of $FGe_1$ and $FHe_2$ is of odd degree and cannot be a totally definite quaternion algebra. This means that $G\times H$ can never be $F$-critical, unless either $G=1$ or $H=1$. If one of both, say $FG e_1$ is a field, then $FGe_1$ is a Wedderburn component of $F (G/G')$ and hence we can assume that $G$ is an abelian group in (Z) and thus cyclic.

Let $G$ be a Z-group of type (Z)(c) which is $F$-critical. By \cref{SubgruposAD} and the above, necessarily, $G= C_{m}\times (\GEN{a}_{p^t} \rtimes_k \GEN{b}_n$) with $\gcd(m,n)=1=\gcd(m,p)=\gcd(n,p)$. By \eqref{Zinequality} of \cref{remarkZ}, $p$ is odd.


Let $A=C_m\times\GEN{ab^{n/k}}$. Then $E_F(G,A)=G$ and $F Ge_C(G,A,1)=F(\zeta_{mp^tk})*G/A$ is a division algebra for any $C\in\mathcal{C}_F(A)$ because of \cref{minimal}. Assume that $t\neq 1$ and take $N=\GEN{a^p}$. We denote by bars the reduction modulo $N$. Then $\overline{G}=C_m\times(\GEN{\overline{a}}_p\rtimes_k \GEN{\overline{b}}_n)$, $E_F(\overline{G},\overline{A})=\overline{G}$, by \cref{reduction}, and, for any $D\in \mathcal{C}_F(\overline{A})$, $F \overline{G} e_{D}(\overline{G},\overline{A},1)=F(\zeta_{mpk})*G/A$ is naturally embedded in $FGe_C(G,A,1)$. Hence also $F\overline{G}$ has a simple component which is a division algebra of degree $\frac{n}{k}$. Therefore it should be a totally definite quaternion algebra and hence $\frac{n}{k}=2$. This means that the action of $\GEN{\overline{b}}$ on $\GEN{\overline{a}}$ is of order 2 and $\overline{b}^{-1}\overline{a}\overline{b}=\overline{a}^{-1}$. Since $F(\zeta_{mk})$ is contained in the center of $F(\zeta_{mpk})*G/A$, which is totally real, $m=1$ and $k\leq 2$. By \cref{remarkZ} $k\neq 1$. So $k=2$ and $n=4$. Furthermore, $F$ is totally real. Now $G$ is as in case (Z)(b), but by \cref{Zb}, $F$ cannot be totally real in order for $G$ to be $F$-critical. This contradiction tells us that $t=1$.

From now on $G=C_m\times (\GEN{a}_{p} \rtimes_k \GEN{b}_n)$ is $F$-critical and we distinguish two cases, $m\neq 1$ and $m=1$.

Assume first that $G=C_m\times (\GEN{a}_{p} \rtimes_k \GEN{b}_n)$ is $F$-critical and $m\neq 1$. Let $N=C_m$ and $A=C_m\times\GEN{ab^{n/k}}$. By bar we denote reduction modulo $N = C_m$. Then $\overline{G}=\GEN{\overline{a}}_p\rtimes_k \GEN{\overline{b}}_n$, $E_F(\overline{G},\overline{A})=\overline{G}$ by \cref{reduction} and, for any $D\in \mathcal{C}_F(\overline{A})$, $F \overline{G} e_{D}(\overline{G},\overline{A},1)=F(\zeta_{pk})*G/A$. It is naturally embedded in $F(\zeta_{mpk})*G/A$ and therefore it is again a division algebra of degree $\frac{n}{k}$. Therefore it should be a totally definite quaternion algebra and hence $\frac{n}{k}=2$. As before, we conclude that $k=2$, $n=4$ and $F$ is totally real. Now for some normal subgroup $M$ of $G$ and some prime divisor $q$ of $m$, ${G/M=C_q \times (\GEN{\overline{a}}_p\rtimes_2 \GEN{\overline{b}}_4)}$. Then $F (G/M) e_{D}(G/M,A/M,1)=F(\zeta_{2qp})*G/A=\quat{-1,(\zeta_{p}-\zeta_{p}^{-1})^2}{F(\zeta_q,\zeta_{p}+\zeta_{p}^{-1})}$ should be a totally definite quaternion algebra, but its center contains $F(\zeta_q)$ which is not totally real. This is a contradiction and hence $m$ is prime. 
So we can assume that $G=C_q \times (\GEN{a}_p\rtimes_2\GEN{b}_4)$ with $q$ and $p$ different odd primes. Using the conditions of (Z)(c), we get $2 = 2o_2(p) \nmid o_q(p)$. Hence $o_q(p)$ is odd. Moreover, if $p\equiv 1 \mod 4$, then $1=v_2(k)\geq v_2(p-1)\geq 2$, a contradiction. Thus $p\equiv -1\mod 4$.
Now \begin{eqnarray*}
 F Ge_C(G,A,1) &=& (F(\zeta_{qp})/F(\zeta_q,\zeta_p+\zeta_p^{-1}),\zeta_p\mapsto \zeta_p^{-1},-1)\\
&=& \quat{-1,(\zeta_{p}-\zeta_{p}^{-1})^2}{F(\zeta_q,\zeta_{p}+\zeta_{p}^{-1})}\\
&=&F(\zeta_q,\zeta_{p}+\zeta_{p}^{-1})\otimes_{\Q(\zeta_{p}+\zeta_{p}^{-1})}\quat{-1,(\zeta_{p}-\zeta_{p}^{-1})^2}{\Q(\zeta_{p}+\zeta_{p}^{-1})}
\end{eqnarray*}
is a division algebra. By \cref{olteanu,splitting}, we have that $m_{p}\quat{-1,(\zeta_{p}-\zeta_{p}^{-1})^2}{\Q(\zeta_{p}+\zeta_{p}^{-1})} \not= 1$ and both $e_{p}(F(\zeta_q,\zeta_{p}+\zeta_{p}^{-1})/\Q(\zeta_{p}+\zeta_{p}^{-1}))$ and $f_{p}(F(\zeta_q,\zeta_{p}+\zeta_{p}^{-1})/\Q(\zeta_{p}+\zeta_{p}^{-1}))$ are odd. 

Since $e=e_p(\Q(\zeta_p)/\Q(\zeta_p+\zeta_p^{-1}))=2$ and $f=f_p(\Q(\zeta_p+\zeta_p^{-1})/\Q)=1$, it follows from \cref{hermanodd} that $m_{p}\quat{-1,(\zeta_{p}-\zeta_{p}^{-1})^2}{\Q(\zeta_{p}+\zeta_{p}^{-1})} \not= 1$ if and only if $p\equiv -1\mod 4$. Since $p\equiv -1\mod 4$, also $e_p(\Q(\zeta_{p}+\zeta_{p}^{-1})/\Q)$, $f_p(\Q(\zeta_{p}+\zeta_{p}^{-1})/\Q)$, $e_{p}(F(\zeta_q,\zeta_{p}+\zeta_{p}^{-1})/F(\zeta_q))$ and $f_{p}(F(\zeta_q,\zeta_{p}+\zeta_{p}^{-1})/F(\zeta_q))$ are odd. Therefore $e_{p}(F(\zeta_q,\zeta_{p}+\zeta_{p}^{-1})/\Q(\zeta_{p}+\zeta_{p}^{-1}))$ and $f_{p}(F(\zeta_q,\zeta_{p}+\zeta_{p}^{-1})/\Q(\zeta_{p}+\zeta_{p}^{-1}))$ are odd if and only if $e_p(F(\zeta_q)/\Q)$ and $f_p(F(\zeta_q)/\Q)$ are odd. This means that $G$ and $F$ are as in \eqref{Zc1}.

Conversely, assume that $G$ and $F$ are as in \eqref{Zc1}, then $G=C_{pq}\rtimes_2 C_4 = C_q \times (\GEN{a}_p \rtimes_2 \GEN{b}_4)$ is a Z-group of type (Z)(c). Let $A=C_q\times\GEN{ab^2}$ and $C\in \mathcal{C}_F(A)$. Since $F$ is totally real, clearly $\Q(\zeta_{p})\cap F \subseteq \Q(\zeta_{p}+\zeta_{p}^{-1})$ and hence $G=E_F(G,A)$. Since $FGe_C(G,A,1)$ has a totally imaginary center $F(\zeta_{q},\zeta_{p}+\zeta_{p}^{-1})$, it is not a totally definite quaternion algebra. Due to $p\equiv -1 \mod 4$, $m_{p}\quat{-1,(\zeta_{p}-\zeta_{p}^{-1})^2}{\Q(\zeta_{p}+\zeta_{p}^{-1})} \not= 1$ and because of the assumptions both $e_{p}(F(\zeta_q,\zeta_{p}+\zeta_{p}^{-1})/\Q(\zeta_{p}+\zeta_{p}^{-1}))$ and $f_{p}(F(\zeta_q,\zeta_{p}+\zeta_{p}^{-1})/\Q(\zeta_{p}+\zeta_{p}^{-1}))$ are odd. Because of \cref{splitting}. $FGe_C(G,A,1)$ is now a division algebra. Therefore $FG$ has an exceptional component of type 1. The proper non-abelian quotients of $G$ are $C_p\rtimes_2 C_4, D_{2p}, D_{2pq}$ and, since $F$ is totally real, those groups give rise to simple components which are either fields, matrix rings or totally definite quaternion algebras. Therefore $G$ is $F$-critical.

Assume now that $m=1$ and $G=\GEN{a}_p \rtimes_k \GEN{b}_n$ is $F$-critical, with $b^{-1}ab=a^r$, $p$ an odd prime and $\gcd(p,n)=1$. By \cref{remarkZ}, if $q$ is a prime divisor of $n$ then $1\le v_q(\frac{n}{k})\le v_q(p-1)\le v_q(k)$. In particular $v_q(n)\ge 2$, so either $n=4$ and $G=C_p\rtimes_2 C_4$ or $n\ge 8$. However $G=C_p\rtimes_2 C_4$ is dealt with in \cref{Zb}, thus we can assume that $n\ge 8$. Let $A=\GEN{ab^{n/k}}$. By assumption, $FG$ has an exceptional component of type 1, so by \cref{minimal}, $FGe_C(G,A,1)$ is a division algebra which has degree $\frac{n}{k}$, and in particular $E_F(G,A)=G$. By \cref{E}, $E_F(G,A)=G$ is equivalent to 
\begin{eqnarray*}
 \Q(\zeta_p)\cap F\ \subseteq\ \Fix(\Q(\zeta_p)\rightarrow \Q(\zeta_p):\zeta_p\mapsto \zeta_p^r)
\ =\ \Q(\zeta_p+\zeta_p^r+...+\zeta_p^{r^{\frac{n}{k}-1}}).
\end{eqnarray*}

From the conditions on (Z)(c), if $n$ is even, we have $2\le v_2(p-1)\le v_2(k)$ when $p\equiv 1 \mod 4$ and either $v_2(k)=1$ or $3\le v_2(p+1)+1\le v_2(k)$ when $p\equiv -1 \mod 4$. Due to $\frac{n}{k}$ divides $p-1$ (see \cref{remarkZ}), if $p\equiv -1 \mod 4$, then $v_2(p-1)=1$ and $v_2(n)=v_2(k)+1$. 
Together with the other conditions form (Z)(c) on the parameters of $G$, this gives rise to the following cases:
\begin{enumerate}[label=(\Roman*),ref=\Roman*]
\item\label{case1} either $p\equiv 1 \mod 4$ or $n$ is odd, $v_q(p-1)\le v_q(k)$ for every prime divisor $q$ of $n$;
\item\label{case2} $p\equiv -1 \mod 4$, $v_2(k)=1$, $v_2(n)=2$ and $v_q(p-1)\le v_q(k)$ for every odd prime divisor $q$ of $n$;
\item\label{case3} $p\equiv -1 \mod 4$, $3\le v_2(p+1)+1\le v_2(k)$, $v_2(n)=v_2(k)+1$ and $v_q(p-1)\le v_q(k)$ for every odd prime divisor $q$ of $n$.
\end{enumerate}

Let $\sigma$ be the automorphism of $F(\zeta_{pk})$ which maps $\zeta_{p}$ to $\zeta_{p}^{r}$ and fixes $F(\zeta_k)$. Let $K=\Fix(\sigma)$. By our assumptions and \cref{structure}, for any $C\in \mathcal{C}_F(A)$ we have $FGe_C(G,A,1)=(K(\zeta_p)/K,\sigma, \zeta_k)$. Since $FGe_C(G,A,1)$ is a division algebra, \[\frac{n}{k}=\ind(FGe_C(G,A,1))=\lcm(m_q(FGe_C(G,A,1)) : q=p,\infty),\] by \cref{lcm,olteanu}. If $m_{\infty}(FGe_C(G,A,1))= 2$, then $K\subseteq \R$ by \cref{hermaninfty}. Note that always $\zeta_k\in K$, so $k=2$. This implies that $n$ is a power of $2$. 
However, when $k=2$, $G$ is in case \eqref{case2} and $v_2(n)=2$, a contradiction because of $n\ge 8$. Thus, ${m_{\infty}(FGe_C(G,A,1))=1}$. By \cref{hermanodd}, $m_p(FGe_C(G,A,1))=m_p$ as in the statement of the theorem. So ${\frac{n}{k}=\ind(FGe_C(G,A,1))=m_p}$ and $FGe_C(G,A,1)$ is an exceptional component of $FG$.

For each of the cases \eqref{case1}-\eqref{case3}, let $h\mid k$, $h\neq k$ be an integer satisfying the conditions as in the statement of the theorem or $h=2$ when $k$ is even, and let $N=\GEN{b^{\frac{hn}{k}}}\subseteq Z(G)$. By $\overline{\phantom{n}}$ we denote reduction modulo $N$. Then $\overline{G}=C_p\rtimes_h C_{\frac{hn}{k}}$ is a non-abelian proper quotient of $G$. Note also that the prime divisors of $n$ and $\frac{hn}{k}$ are the same since the prime divisors of $n$ and $\frac{n}{k}$ are the same. Since $N\subseteq Z(G)$, the images of the actions of $C_n$ on $C_p$ and of $C_{\frac{hn}{k}}$ on $C_p$ are the same. So each Sylow $q$-subgroup of $C_{\frac{hn}{k}}$ acts non-trivial on $C_p$ and $\overline{G}$ is again a Z-group of type (Z)(b) or (Z)(c). By our assumption on $G$, $\overline{G}$ can not have exceptional simple components. Therefore the simple component $F\overline{G}e_D(\overline{G},\overline{A},1)$ of degree $\frac{n}{k}$ of $F\overline{G}$ is not exceptional for any $D\in \mathcal{C}_F(\overline{A})$. Since $E_F(G,A)=G$, by \cref{reduction}, $E_F(\overline{G},\overline{A})=\overline{G}$, and then by \cref{structure}, $F\overline{G}e_D(\overline{G},\overline{A},1)=(K_h(\zeta_p)/K_h,\sigma, \zeta_h)$ with $\zeta_h\in K_h$. We also use $\sigma$ to denote its restriction to $F(\zeta_{ph})$. Hence \[ \ind(F\overline{G}e_D(\overline{G},\overline{A},1))=\lcm(m_q(F\overline{G}e_D(\overline{G},\overline{A},1)) : q=p,\infty)\le \frac{n}{k}\] by \cref{lcm,olteanu} and $F\overline{G}e_D(\overline{G},\overline{A},1)$ is either a matrix ring or a totally definite quaternion algebra. If $F\overline{G}e_D(\overline{G},\overline{A},1)$ is a totally definite quaternion algebra, then the degree $\frac{n}{k}=2$ and $K_h\subseteq \R$, so $h=2$ and $\overline{G}=C_p\rtimes_2 C_4$. In this case $F$ is totally real. If $F\overline{G}e_D(\overline{G},\overline{A},1)$ is a matrix ring, then $\lcm(m_q(F\overline{G}e_D(\overline{G},\overline{A},1)) : q=p,\infty)< \frac{n}{k}$. Furthermore, suppose that  $m_{\infty}(F\overline{G}e_D(\overline{G},\overline{A},1))=2$, then $K_h\subseteq \R$ by \cref{hermaninfty}, thus $h=2$. It follows that $\frac{hn}{k}=4$ and $\frac{n}{k}=2$. So $\lcm(m_q(F\overline{G}e_D(\overline{G},\overline{A},1)) : q=p,\infty)=1$, and hence ${m_{\infty}(F\overline{G}e_D(\overline{G},\overline{A},1))=1}$, a contradiction. Thus $m_{\infty}(F\overline{G}e_D(\overline{G},\overline{A},1))=1$ and hence $m_p(F\overline{G}e_D(\overline{G},\overline{A},1)=m_{p,h}<\frac{n}{k}$, with $m_{p,h}$ the formula as in the statement of the theorem (\cref{hermanodd}). 

Suppose that $h\neq 2$, then $K_h$ is not totally real and hence $F\overline{G}e_D(\overline{G},\overline{A},1)$ is a matrix ring and $m_{p,h}<\frac{n}{k}$. If $h=2$, then $n=2k$ and $\overline{G}=C_p\rtimes_2 C_4$. If $F\overline{G}e_D(\overline{G},\overline{A},1)$ is a matrix ring, then $m_{p,2}<2$, so $m_{p,2}=1$. If $F\overline{G}e_D(\overline{G},\overline{A},1)$ is a totally definite quaternion algebra, then $F$ is totally real. Hence $G$ is as in \eqref{Zc2}.

Assume now that $G$ and $F$ are as in \eqref{Zc2}. Since any prime divisor $q$ of $n$ divides $\frac{n}{k}$, the order of the image of the action of $C_n$ on $C_p$, any Sylow $q$-subgroup of $C_n$ acts non-trivial on $C_p$. Together with the assumptions on $n$ and $k$ in \eqref{zc1}-\eqref{zc3}, this means that $G$ is a Z-group of type (Z)(c). Let $A=\GEN{ab^{n/k}}$. By \cref{E} and the assumptions on $F$, $E_F(G,A)=G$. Then by \cref{structure}, $FGe_C(G,A,1)=(K(\zeta_p)/K,\sigma, \zeta_k)$ for any $C\in \mathcal{C}_F(A)$. Furthermore, $FGe_C(G,A,1)$ has degree $\frac{n}{k}$ and always $\zeta_k\in K$. By \cref{lcm,olteanu}, \[ \ind(FGe_C(G,A,1))=\lcm(m_q(FGe_C(G,A,1)) : q=p,\infty).\] If $m_{\infty}(FGe_C(G,A,1))=2$, then $K\subseteq \R$ by \cref{hermaninfty}, and hence $k=2$ and $n$ is a power of 2. So $G$ is as in \eqref{zc1} or \eqref{zc3}, but from both conditions we can deduce that $k$ can not be $2$. Thus, $m_{\infty}(FGe_C(G,A,1))=1$. Therefore $\ind(FGe_C(G,A,1))=m_p(FGe_C(G,A,1))$ and by \cref{hermanodd} and the assumptions, we have $m_p(FGe_C(G,A,1))=m_p=\frac{n}{k}$ and $FGe_C(G,A,1)$ is an exceptional component of type 1.

In order to prove that $G$ does not have proper quotients with exceptional components of type 1 in their Wedderburn decomposition over $F$, we argue by means of contradiction. Let $N$ be a normal subgroup of $G$ such that $F(G/N)$ has an exceptional component of type 1. We will use bars for reduction modulo $N$. Then $\overline{G}$ is non-abelian and hence $N=\GEN{b^{\frac{hn}{k}}}\subseteq Z(G)$ for some divisor $h$ of $k$, $h\neq k$. Then $\overline{G}=C_p\rtimes_h C_{\frac{hn}{k}}$ and without loss of generality we can assume that this group has a faithful exceptional component of type 1. Then $\overline{G}$ is as in (Z)(b) or (Z)(c), and by \cref{unique}, $F\overline{G}e_D(\overline{G},\overline{A},1)=(K_h(\zeta_p)/K_h,\sigma ,\zeta_h)$ is an exceptional division algebra. Hence ${\frac{n}{k}=\ind(F\overline{G}e_D(\overline{G},\overline{A},1))=\lcm(m_q(F\overline{G}e_D(\overline{G},\overline{A},1)) : q=p,\infty)}$. 

If $\overline{G}$ is as in (Z)(b), then $\overline{G}=C_p\rtimes_2 C_4$, $h=2$ and $n=2k$. So $G$ is as in \eqref{zc1} or \eqref{zc3} and by the assumptions on $G$, $F$ is totally real or $m_{p,2}=1$. By \cref{structure}, \[F\overline{G}e_D(\overline{G},\overline{A},1)=\quat{-1,(\zeta_p-\zeta_p^{-1})^2}{F(\zeta_p+\zeta_p^{-1})}.\] We assume that it is exceptional, so $\lcm(m_q(F\overline{G}e_D(\overline{G},\overline{A},1)) : q=p,\infty)=2$. If $F$ is totally real, then $\quat{-1,(\zeta_p-\zeta_p^{-1})^2}{F(\zeta_p+\zeta_p^{-1})}$ is a totally definite quaternion algebra, a contradiction. Hence $m_p(F\overline{G}e_D(\overline{G},\overline{A},1))=m_{p,2}=1$, but then $m_{\infty}(F\overline{G}e_D(\overline{G},\overline{A},1))=2$ and $F\overline{G}e_D(\overline{G},\overline{A},1)$ is not an exceptional component, again a contradiction.

Hence $\overline{G}$ is as in (Z)(c). If $m_{\infty}(F\overline{G}e_D(\overline{G},\overline{A},1))=2$, then $K_h\subseteq \R$ by \cref{hermaninfty}, and then $h=2$. However, $h=2$ implies that $n$ is a power of $2$, $p\equiv -1 \mod 4$ and $v_2(\frac{hn}{k})=2$. It follows that $\frac{hn}{k}=4$ and $\overline{G}=C_p\rtimes_2 C_4$, a contradiction. So, $m_{\infty}(F\overline{G}e_D(\overline{G},\overline{A},1))=1$. Suppose that $p\equiv 1 \mod 4$ or $\frac{hn}{k}$ is odd (equivalently $n$ is odd; case \eqref{zc1}), then $v_q(p-1)\le v_q(h)$ for all prime divisors $q$ of $\frac{hn}{k}$ (which are exactly the prime divisors of $n$). Assume first that $h\neq 2$, then by the assumptions, $m_p(F\overline{G}e_D(\overline{G},\overline{A},1))=m_{p,h}<\frac{n}{k}$, a contradiction. Now regard the case when $h=2$, then necessarily $n$ is a power of $2$ and $\overline{G}=C_p\rtimes_2 C_4$, a contradiction. Assume that $p\equiv -1 \mod 4$ and $\frac{hn}{k}$ is even (equivalently $n$ is even). Then $v_q(p-1)\le v_q(h)$ for all odd prime divisors $q$ of $n$. Also, either $v_2(k)=1$ or $v_2(p+1)+1\le v_2(k)$. We first deal with $v_2(k)=1$ (case \eqref{zc2}). In this case, $v_2(n)=2$, $v_2(h)=1$ and $h\ne 2$ (as otherwise $n=4$). So $m_p(F\overline{G}e_D(\overline{G},\overline{A},1))<\frac{n}{k}$, a contradiction. Finally, suppose that $v_2(p+1)+1\le v_2(k)$ (case \eqref{zc3}). Then either $v_2(p+1)+1\le v_2(h)$ or $v_2(h)=1$. If $h\neq 2$, then $m_p(F\overline{G}e_D(\overline{G},\overline{A},1))<\frac{n}{k}$, a contradiction. When $h=2$, $\overline{G}=C_p\rtimes_2 C_4$ gives a contradiction. 

We conclude that $G$ does not have proper quotients with exceptional components of type 1, hence $G$ is $F$-critical.
\end{proof}

\begin{theorem}\label{classification_division}
 Let $D$ be a division ring and $F$ an abelian number field, $p$ and $q$ different odd prime numbers. Then $D$ is a Wedderburn component of $FG$ for an $F$-critical group $G$ if and only if one of the following holds:
 \begin{enumerate}[label=(\alph*),ref=\alph*]
  \item $D = \quat{-1,-1}{F}$, $G \in \{ \SL(2,3), Q_8 \}$, $F$ is totally imaginary and both, $e_2(F/\Q)$ and $f_2(F/\Q)$, are odd;
  \item $D = \quat{-1,-1}{F(\zeta_p)}$, $G \in \{\SL(2,3) \times C_p, Q_8 \times C_p\}$, $\gcd(p, |G|/p) = 1$, $o_p(2)$ is odd, $F$ is totally real and both, $e_2(F(\zeta_p)/\Q)$ and $f_2(F(\zeta_p)/\Q)$, are odd;
  \item $D = \quat{-1,(\zeta_p - \zeta_p^{-1})^2}{F(\zeta_p + \zeta_p^{-1})}$, $G = C_p \rtimes_2 C_4$, $p \equiv -1 \mod 4$, $F$ totally imaginary, $\Q(\zeta_p) \cap F \subseteq \Q(\zeta_p + \zeta_p^{-1})$ and both, $e_p(F/\Q)$ and $f_p(F/\Q)$, are odd;
  \item $D = \quat{-1,(\zeta_p - \zeta_p^{-1})^2}{F(\zeta_q, \zeta_p + \zeta_p^{-1})}$, $G = C_q \times (C_p \rtimes_2 C_4)$, $p \equiv -1 \mod 4$, $o_q(p)$ odd, $F$ is totally real and both, $e_p(F(\zeta_q)/\Q)$ and $f_p(F(\zeta_q)/\Q)$, are odd;
  \item $D = (K(\zeta_{p})/K,\sigma,\zeta_{k})$ with Schur index $\frac{n}{k}$, $G=\GEN{a}_p\rtimes_k \GEN{b}_n$ with $n\geq 8$, $\gcd(p, n) = 1$, $b^{-1}ab=a^r$, and both $k$ and $\frac{n}{k}$ are divisible by all the primes dividing $n$. Here $K=F(\zeta_k,\zeta_p+\zeta_p^r+...+\zeta_p^{r^{\frac{n}{k}-1}})$ and $\sigma: F(\zeta_{pk})\rightarrow F(\zeta_{pk}):\zeta_p\mapsto \zeta_p^r;\zeta_k\mapsto \zeta_k$. Moreover $\Q(\zeta_p)\cap F\subseteq \Q(\zeta_p+\zeta_p^r+...+\zeta_p^{r^{\frac{n}{k}-1}})$ and one of the conditions \eqref{zc1} - \eqref{zc3} from \cref{Zc} holds. Furthermore $$\min\left\{l\in\N \left| \frac{p^f-1}{\gcd(p^f-1,e)} \equiv 0 \mod \frac{k}{\gcd(k,l)}\right.\right\} = \frac{n}{k}$$
with $e=e_p(F(\zeta_{pk})/K)$ and $f=f_p(K/\Q)$.
 \end{enumerate}
\end{theorem}

\begin{proof} This follows immediately by combining \cref{SubgruposAD,SL23,O,quaternion,Zb,Zc}. \end{proof}

\section{Examples}

We apply our results of \cref{classification_division} to the prominent case of cyclotomic field $\Q(\zeta_m)$. Because of knowledge on the ramification index and residue degree of these extensions the conditions simplify.

\begin{corollary}
 Let $D$ be a division ring, $F=\Q(\zeta_m)$ a cyclotomic field, $m = 2^sm' \ge 3$ with $\gcd(m',2) = 1$ and $p$ an odd prime number. Then $D$ is a Wedderburn component of $FG$ for an $F$-critical group $G$ if and only if one of the following holds:
 \begin{enumerate}[label=(\alph*),ref=\alph*]
  \item $D = \quat{-1,-1}{F}$, $G \in \{ \SL(2,3), Q_8 \}$, $s\in\{0,1\}$ and $o_{m'}(2)$ is odd;
  \item $D = \quat{-1,(\zeta_p - \zeta_p^{-1})^2}{F(\zeta_p + \zeta_p^{-1})}$, $G = C_p \rtimes_2 C_4$, $p \equiv -1 \mod 4$, $\gcd(p,m) = 1$ and $o_m(p)$ is odd;
  \item $D = (K(\zeta_{p})/K,\sigma,\zeta_{k})$ with Schur index $\frac{n}{k}$, $G=\GEN{a}_p\rtimes_k \GEN{b}_n$ with $n\geq 8$, $\gcd(p, n) = \gcd(p ,m) =  1$, $b^{-1}ab=a^r$, and both $k$ and $\frac{n}{k}$ are divisible by all the primes dividing $n$. Here $K=F(\zeta_k,\zeta_p+\zeta_p^r+...+\zeta_p^{r^{\frac{n}{k}-1}})$ and $\sigma: F(\zeta_{pk})\rightarrow F(\zeta_{pk}):\zeta_p\mapsto \zeta_p^r;\zeta_k\mapsto \zeta_k$. Moreover one of the conditions \eqref{zc1} - \eqref{zc3} from \cref{Zc} holds. Furthermore $$\min\left\{l\in\N \left| \frac{p^f-1}{\gcd(p^f-1,e)} \equiv 0 \mod \frac{k}{\gcd(k,l)}\right.\right\} = \frac{n}{k}$$
with $e=\frac{n}{k}$ and $f=o_{mk}(p)$.
 \end{enumerate}
\end{corollary}



Note that from \cref{exceptional_table} one can easily derive all finite groups $G$ and simple exceptional algebras $B$ of type 2, such that $\Q(\zeta_m) G$ ($m \geq 3$) has $B$ in its Wedderburn decomposition. In such a case $m$ is restricted to $3$ and $4$.

The conditions of \cref{classification_division} are easy to check algorithmically and we did implement it in \texttt{GAP}. With this programm we can compute the $F$-critical groups for any abelian number field $F$ up to a fixed order. As an illustration we include the $F$-critical groups up to order 200 for all subfields of $\Q(\zeta_7)$. We compute the Schur index of the corresponding exceptional component $A$ and we denote the center of $A$ in the standard \texttt{GAP} notation. A local Schur index $[p,s]$ means that $m_p(A)=s$ and $m_q(A)=1$ for all other primes $q$. When for a fixed group, there are multiple lines in the table, this means that the exceptional component $A$ appears as several isomorphic copies in the Wedderburn decomposition.

{\small 
\begin{longtable}{@{}lllcl@{}} \caption{List of $\Q(\zeta_7)$-critical groups of type $1$ up to order 200} \label{exceptional_table_over_CF(7)} \\ \toprule[1.5pt]
\textsc{SmallGroup} ID & Structure & Center & Schur index & Local index \\ \midrule 
\endfirsthead \toprule[1.5pt] \textsc{SmallGroup} ID & Structure & Center & Schur index & Local index \\ \midrule 
\endhead \hline \multicolumn{5}{c}{continued}\\ \midrule[1.5pt]\endfoot\bottomrule[1.5pt]\endlastfoot
{[}8, 4{]} & $ Q_8 $ & CF(7) & 2 & {[}2, 2{]} \\
{[}24, 3{]} & $ {\rm SL} (2,3)$ & CF(7) & 2 & {[}2, 2{]} \\
{[}44, 1{]} & $C_{11}  \rtimes_{2} C_{4}$ & NF(77,[ 1, 43 ]) & 2 & {[}11, 2{]} \\
{[}48, 1{]} & $C_{3}  \rtimes_{8} C_{16}$ & CF(56) & 2 & {[}3, 2{]} \\
{[}80, 1{]} & $C_{5}  \rtimes_{8} C_{16}$ & NF(280,[ 1, 169 ]) & 2 & {[}5, 2{]} \\
{[}92, 1{]} & $C_{23}  \rtimes_{2} C_{4}$ & NF(161,[ 1, 22 ]) & 2 & {[}23, 2{]} \\
{[}117, 1{]} & $C_{13}  \rtimes_{3} C_{9}$ & NF(273,[ 1, 22, 211 ]) & 3 & {[}13, 3{]} \\
{[}160, 3{]} & $C_{5}  \rtimes_{8} C_{32}$ & CF(56) & 4 & {[}5, 4{]} \\
{[}172, 1{]} & $C_{43}  \rtimes_{2} C_{4}$ & NF(301,[ 1, 85 ]) & 2 & {[}43, 2{]} \\
\end{longtable}}
{\small 
\begin{longtable}{@{}lllcl@{}} \caption{List of $\Q(\zeta_7 + \zeta_7^{-1})$-critical groups of type $1$ up to order 200} \label{exceptional_table_over_NF(7,[ 1, 6 ])} \\ \toprule[1.5pt]
\textsc{SmallGroup} ID & Structure & Center & Schur index & Local index \\ \midrule 
\endfirsthead \toprule[1.5pt] \textsc{SmallGroup} ID & Structure & Center & Schur index & Local index \\ \midrule 
\endhead \hline \multicolumn{5}{c}{continued}\\ \midrule[1.5pt]\endfoot\bottomrule[1.5pt]\endlastfoot
{[}40, 1{]} & $C_{5}  \rtimes_{4} C_{8}$ & NF(140,[ 1, 29, 41, 69 ]) & 2 & {[}5, 2{]} \\
{[}48, 1{]} & $C_{3}  \rtimes_{8} C_{16}$ & NF(56,[ 1, 41 ]) & 2 & {[}3, 2{]} \\
{[}56, 10{]} & $C_{7}  \times   Q_8 $ & CF(7) & 2 & {[}2, 2{]}\\
 & & CF(7) & 2 & {[}2, 2{]}\\
 & & CF(7) & 2 & {[}2, 2{]} \\
{[}80, 3{]} & $C_{5}  \rtimes_{4} C_{16}$ & NF(28,[ 1, 13 ]) & 4 & {[}5, 4{]} \\
{[}84, 4{]} & $C_{3}  \times  (C_{7}  \rtimes_{2} C_{4})$ & NF(21,[ 1, 13 ]) & 2 & {[}7, 2{]}\\
 & & NF(21,[ 1, 13 ]) & 2 & {[}7, 2{]}\\
 & & NF(21,[ 1, 13 ]) & 2 & {[}7, 2{]} \\
{[}104, 1{]} & $C_{13}  \rtimes_{4} C_{8}$ & NF(364,[ 1, 181, 209, 337 ]) & 2 & {[}13, 2{]} \\
{[}117, 1{]} & $C_{13}  \rtimes_{3} C_{9}$ & NF(273,[ 1, 22, 55, 118, 139, 211 ]) & 3 & {[}13, 3{]} \\
{[}132, 1{]} & $C_{11}  \times  (C_{3}  \rtimes_{2} C_{4})$ & NF(77,[ 1, 34 ]) & 2 & {[}3, 2{]} \\
{[}156, 3{]} & $C_{13}  \times  (C_{3}  \rtimes_{2} C_{4})$ & NF(91,[ 1, 27 ]) & 2 & {[}3, 2{]} \\
{[}168, 22{]} & $C_{7}  \times   {\rm SL} (2,3)$ & CF(7) & 2 & {[}2, 2{]}\\
 & & CF(7) & 2 & {[}2, 2{]}\\
 & & CF(7) & 2 & {[}2, 2{]} \\
{[}176, 1{]} & $C_{11}  \rtimes_{8} C_{16}$ & NF(616,[ 1, 153, 265, 505 ]) & 2 & {[}11, 2{]} \\
{[}184, 10{]} & $C_{23}  \times   Q_8 $ & NF(161,[ 1, 139 ]) & 2 & {[}2, 2{]} \\
\end{longtable}}
{\footnotesize 
\begin{longtable}{@{}lllcl@{}} \caption{List of $\Q(\sqrt{-7})$-critical groups of type $1$ up to order 200} \label{exceptional_table_over_NF(7,[ 1, 2, 4 ])} \\ \toprule[1.5pt]
\textsc{SmallGroup} ID & Structure & Center & Schur index & Local index \\ \midrule 
\endfirsthead \toprule[1.5pt] \textsc{SmallGroup} ID & Structure & Center & Schur index & Local index \\ \midrule 
\endhead \hline \multicolumn{5}{c}{continued}\\ \midrule[1.5pt]\endfoot\bottomrule[1.5pt]\endlastfoot
{[}8, 4{]} & $ Q_8 $ & NF(7,[ 1, 2, 4 ]) & 2 & {[}2, 2{]} \\
{[}24, 3{]} & $ {\rm SL} (2,3)$ & NF(7,[ 1, 2, 4 ]) & 2 & {[}2, 2{]} \\
{[}44, 1{]} & $C_{11}  \rtimes_{2} C_{4}$ & NF(77,[ 1, 23, 32, 43, 65, 67 ]) & 2 & {[}11, 2{]} \\
{[}48, 1{]} & $C_{3}  \rtimes_{8} C_{16}$ & NF(56,[ 1, 9, 25 ]) & 2 & {[}3, 2{]} \\
{[}63, 1{]} & $C_{7}  \rtimes_{3} C_{9}$ & NF(21,[ 1, 4, 16 ]) & 3 & {[}7, 3{]}\\
 & & NF(21,[ 1, 4, 16 ]) & 3 & {[}7, 3{]} \\
{[}80, 1{]} & $C_{5}  \rtimes_{8} C_{16}$ & NF(280,[ 1, 9, 81, 121, 169, 249 ]) & 2 & {[}5, 2{]} \\
{[}92, 1{]} & $C_{23}  \rtimes_{2} C_{4}$ & NF(161,[ 1, 22, 93, 114, 116, 137 ]) & 2 & {[}23, 2{]} \\
{[}117, 1{]} & $C_{13}  \rtimes_{3} C_{9}$ & NF(273,[ 1, 16, 22, 79, 100, 172, 211, 235, 256 ]) & 3 & {[}13, 3{]} \\
{[}160, 3{]} & $C_{5}  \rtimes_{8} C_{32}$ & NF(56,[ 1, 9, 25 ]) & 4 & {[}5, 4{]} \\
{[}172, 1{]} & $C_{43}  \rtimes_{2} C_{4}$ & NF(301,[ 1, 44, 85, 128, 130, 214 ]) & 2 & {[}43, 2{]} \\
\end{longtable}}
{\small 
\begin{longtable}{@{}lllcl@{}} \caption{List of $\Q$-critical groups of type $1$ up to order 200} \label{exceptional_table_over_Rationals} \\ \toprule[1.5pt]
\textsc{SmallGroup} ID & Structure & Center & Schur index & Local index \\ \midrule 
\endfirsthead \toprule[1.5pt] \textsc{SmallGroup} ID & Structure & Center & Schur index & Local index \\ \midrule 
\endhead \hline \multicolumn{5}{c}{continued}\\ \midrule[1.5pt]\endfoot\bottomrule[1.5pt]\endlastfoot
{[}40, 1{]} & $C_{5}  \rtimes_{4} C_{8}$ & NF(20,[ 1, 9 ]) & 2 & {[}5, 2{]} \\
{[}48, 1{]} & $C_{3}  \rtimes_{8} C_{16}$ & CF(8) & 2 & {[}3, 2{]} \\
{[}56, 10{]} & $C_{7}  \times   Q_8 $ & CF(7) & 2 & {[}2, 2{]} \\
{[}63, 1{]} & $C_{7}  \rtimes_{3} C_{9}$ & NF(21,[ 1, 4, 16 ]) & 3 & {[}7, 3{]} \\
{[}80, 3{]} & $C_{5}  \rtimes_{4} C_{16}$ & GaussianRationals & 4 & {[}5, 4{]} \\
{[}84, 4{]} & $C_{3}  \times  (C_{7}  \rtimes_{2} C_{4})$ & NF(21,[ 1, 13 ]) & 2 & {[}7, 2{]} \\
{[}104, 1{]} & $C_{13}  \rtimes_{4} C_{8}$ & NF(52,[ 1, 25 ]) & 2 & {[}13, 2{]} \\
{[}117, 1{]} & $C_{13}  \rtimes_{3} C_{9}$ & NF(39,[ 1, 16, 22 ]) & 3 & {[}13, 3{]} \\
{[}132, 1{]} & $C_{11}  \times  (C_{3}  \rtimes_{2} C_{4})$ & CF(11) & 2 & {[}3, 2{]} \\
{[}156, 3{]} & $C_{13}  \times  (C_{3}  \rtimes_{2} C_{4})$ & CF(13) & 2 & {[}3, 2{]} \\
{[}168, 22{]} & $C_{7}  \times   {\rm SL} (2,3)$ & CF(7) & 2 & {[}2, 2{]} \\
{[}176, 1{]} & $C_{11}  \rtimes_{8} C_{16}$ & NF(88,[ 1, 65 ]) & 2 & {[}11, 2{]} \\
{[}184, 10{]} & $C_{23}  \times   Q_8 $ & CF(23) & 2 & {[}2, 2{]} \\
\end{longtable}}

\section*{Acknowledgements}
The authors would like to thank Allen Herman and \'Angel del R\'io for the helpful discussions.

\renewcommand{\bibname}{References}
\bibliographystyle{amsalpha}
\bibliography{referencesFebruary2015}

\end{document}